\documentclass[preprint, 3p]{elsarticle}

\usepackage[utf8]{inputenc}
\usepackage[english]{babel}
\usepackage{etoolbox}
\usepackage{amsmath}
\usepackage{amsfonts}
\usepackage{amssymb}
\usepackage{amsthm}
\usepackage{enumitem}
\usepackage{thmtools}
\usepackage{dsfont}
\usepackage{hyperref}

\makeatletter
\def\ps@pprintTitle{%
 \let\@oddhead\@empty
 \let\@evenhead\@empty
 \def\@oddfoot{\centerline{\thepage}}%
 \let\@evenfoot\@oddfoot}
\makeatother

\allowdisplaybreaks

\newcommand{\eps}{\varepsilon}

\newtheoremstyle{slplain}
  {0.4cm}
  {0.4cm}
  {\upshape}
  {}
  {\bfseries}
  {.}
  { }
  {}
  
\newtheoremstyle{itplain}
    {0.4cm}
    {0.4cm}
    {\itshape}
    {}
    {\bfseries}
    {.}
    { }
    {}
  
\declaretheorem[style=slplain,numberwithin=section]{definition}

\declaretheorem[style=slplain,sibling=definition]{remark}

\declaretheorem[style=itplain,sibling=definition]{theorem}

\declaretheorem[style=itplain,sibling=definition]{proposition}
\declaretheorem[style=itplain,sibling=definition]{lemma}

\patchcmd{\MaketitleBox}{\footnotesize\itshape\elsaddress\par\vskip36pt}{\footnotesize\itshape\elsaddress\par\parbox[b][36pt]{\linewidth}{\vfill\hfill\textnormal{May 5, 2017}\hfill\null\vfill}}{}{}%
\patchcmd{\pprintMaketitle}{\footnotesize\itshape\elsaddress\par\vskip36pt}{\footnotesize\itshape\elsaddress\par\parbox[b][36pt]{\linewidth}{\vfill\hfill\textnormal{May 5, 2017}\hfill\null\vfill}}{}{}%

\begin{document}

\begin{frontmatter}

\title{\mbox{ }\\[0.5cm] Barrier Option Pricing under the\\ 2-Hypergeometric Stochastic Volatility Model}

\author[myaddress0,myaddress1]{Rúben Sousa\corref{mycorrespondingauthor}}
\cortext[mycorrespondingauthor]{Corresponding author}
\ead{ruben.azevedo.sousa@tecnico.ulisboa.pt}

\author[myaddress2]{Ana Bela Cruzeiro}
\ead{abcruz@math.tecnico.ulisboa.pt}

\author[myaddress3]{Manuel Guerra}
\ead{mguerra@iseg.ulisboa.pt}

\address[myaddress0]{Centro de Matemática, Faculdade de Ciências, Universidade do Porto, Rua do Campo Alegre 687, 4169-007 Porto, Portugal}
\address[myaddress1]{Instituto Superior Técnico, Universidade de Lisboa, Av.\ Rovisco Pais, 1049-001 Lisbon, Portugal}
\address[myaddress2]{GFMUL and Departamento de Matemática, Instituto Superior Técnico, Universidade de Lisboa,\linebreak Av.\ Rovisco Pais, 1049-001 Lisbon, Portugal}
\address[myaddress3]{CEMAPRE and ISEG (School of Economics and Management), Universidade de Lisboa,\linebreak Rua do Quelhas, 1200-781 Lisbon, Portugal}

\begin{abstract}
We investigate the pricing of financial options under the 2-hypergeometric stochastic volatility model. This is an analytically tractable model that reproduces the volatility smile and skew effects observed in empirical market data.

Using a regular perturbation method from asymptotic analysis of partial differential equations, we derive an explicit and easily computable approximate formula for the pricing of barrier options under the 2-hypergeometric stochastic volatility model. The asymptotic convergence of the method is proved under appropriate regularity conditions, and a multi-stage method for improving the quality of the approximation is discussed. Numerical examples are also provided.
\end{abstract}


\end{frontmatter}

\section{Introduction}

Barrier options are options whose payoff does not depend only on the value of the underlying asset at maturity, but also on whether the path of the asset's price touches a given barrier level during the lifetime of the option. These options, which constitute one of the oldest types of exotic options, have become increasingly popular in the financial derivative industry because they allow for more flexible payoff schemes than plain vanilla options. It is thus important to construct good barrier option pricing models which are able to reproduce the features observed in real market data.

The simplest model for the pricing of financial derivatives is the Black and Scholes model, in which the price of all the standard barrier call and put options can be written in closed form. However, it is widely known that the strong assumptions of this model are unrealistic. In particular, the constant volatility assumption is clearly incompatible with the so-called smile and skew patterns which are generally present in empirical option prices.

A natural way to address this issue is to introduce randomness in the volatility. For this reason, option pricing under stochastic volatility has been the subject of a great deal of research in recent years. Here we focus on the 2-hypergeometric stochastic volatility model, which was introduced by Da Fonseca and Martini \cite{fonseca2015} as a model which ensures that the volatility is strictly positive --- this is an important property which is not present in some other well-established stochastic volatility models. In a very recent paper, Privault and She \cite{privaultshe2016} demonstrated that, under this model, a closed-form asymptotic vanilla option pricing formula can be determined through a regular perturbation method. This is a notable result because their formulas are analytically very simple, which is rarely the case in models with stochastic volatility: as discussed by Zhu \cite{zhu2010}, the higher complexity of these models usually yields the need for rather sophisticated numerical implementations.

The literature on barrier option pricing methods is extensive. Exact closed-form pricing formulas have been derived for only a few models other than that of Black and Scholes, none of which reproduces satisfactorily the market phenomena. (For explicit formulas under one-dimensional  models see e.g.\ Davydov and Linestky \cite{davydov2003}, Hui and Lo \cite{hui2006}; for an explicit solution under the Heston stochastic volatility model see Lipton \cite{lipton2001}.) Given the unavailability of explicit formulas, to price barrier options under more complex models one needs to resort to numerical methods. The main approaches are the use of numerical partial differential equation (PDE) techniques and of Monte Carlo methods (we refer the reader to the books of Seydel \cite{seydel2012} and of Glasserman \cite{glasserman2004}), which are often combined with other analytical or numerical techniques (for recent work see for instance Zhang et al.\ \cite{zhang2014}, Guardasoni and Sanfelici \cite{guardasoni2016}). Unfortunately, the computational times are nowadays still largely incompatible with the demands of the financial industry.

An alternative strategy for pricing under more general models is to derive approximate (or asymptotic) analytical solutions: this has been proposed not only for vanilla options (cf.\ Privault and She \cite{privaultshe2016}) but also for barrier options, see e.g.\ Fouque et al.\ \cite{fouque2000}, Alos et al.\ \cite{alos2016}. These asymptotic methods are intrinsically computationally much less expensive than the numerical methods mentioned in the previous paragraph. Indeed, numerical PDE techniques usually rely on space-time discretization and on the numerical solution of linear systems of high dimension, while Monte Carlo methods require the simulation of a large number of sample paths on a suitably fine time grid; on the other hand, the asymptotic techniques only require the computation of a few integrals (the number of such integrals is small and does not depend on the discretization). Thus the key question when dealing with asymptotic solutions is whether they are sufficiently exact for practical purposes.

Despite the vast body of work in this area, the pricing of barrier and other exotic options under the 2-hypergeometric model has to our knowledge never been studied. Motivated by this, we extend the regular perturbation approach of Privault and She in order to derive an asymptotic pricing formula for barrier-type options. We show that, for a given class of nonconstant barrier functions, an explicit asymptotic formula can be obtained and its convergence can be proved with the help of the Feynman-Kac theorem for Cauchy-Dirichlet problems for parabolic PDEs. Given that in general our class of barrier functions does not include constant functions, the choice of a nonconstant barrier function which approximates a certain constant barrier level is discussed. We also present some numerical examples which indicate that the accuracy of our asymptotic formulas is high enough for the applications.

This paper is organized as follows. In Section \ref{chap:barrier_stochvolprice} we introduce the class of barrier options which we consider, and we present the formulation of the barrier option pricing problem under a generic stochastic volatility model. Section \ref{chap:main}, the main section of this paper, develops the asymptotic pricing approach for barrier-type options: the first-order small volatility expansion is carried out in Subsection \ref{sec:expansion}, the explicit expressions for the zero and first-order terms are derived in Subsections \ref{sec:zeroorder} and \ref{sec:firstorder} respectively, the proof of the convergence of the asymptotic solution is provided in Subsection \ref{sec:conv_proof}, and a generalization of the method to a wider class of barrier functions is given in Subsection \ref{sec:multistage}. In Section \ref{chap:numerics} we present some numerical results to corroborate our theoretical findings. Section \ref{chap:conclusions} summarizes the main conclusions. The appendices contain some auxiliary results.

\section{Barrier option pricing under stochastic volatility} \label{chap:barrier_stochvolprice}

This work focuses on the pricing of \emph{down-and-out call} (DOC) options, which are one of the eight types of standard barrier options. The techniques used in this paper may also be applied to other types of barrier options, such as options with up barriers or put payoffs. The payoff of a DOC option with maturity $T$ is
\[
(S_T - K)^+\, \mathds{1}_{\{S_t > H\text{ for all }0 \leq t \leq T\}},
\]
i.e, it has the usual vanilla call payoff if the asset price process $S$ does not go below the barrier $H$ during the lifetime of the option, and it is worthless otherwise. The DOC option is said to be \textit{regular} if $K \geq H$ and \textit{reverse} if $K < H$. If a barrier function $H(t)$ is considered instead of a constant barrier $H$, the DOC option is called \textit{time-dependent}.

For the sake of generality, let us begin by assuming that the asset process is governed (under the physical measure $\mathbb{P}$) by a Markovian stochastic volatility model of the form
\begin{equation} \label{eq:stochvol_genSDE}
\begin{aligned}
d S_t & = \mu(t,S_t) S_t\, dt + g(V_t) S_t\, dW_t^1 \\[-0.52mm]
d V_t & = a(t,V_t)\, dt + b(t,V_t)\, d W_t^2
\end{aligned}
\end{equation}
where $S$ is the asset price process, $V$ is the volatility process, $W^1$ and $W^2$ are Brownian motions with correlation $\rho \neq \pm 1$, and $g$ is a smooth, positive and increasing function. This is a general family of models which includes the 2-hypergeometric model addressed in the main section of this paper, as well as the Heston model and other popular stochastic volatility models.

It is worth stressing that, unlike the Black and Scholes model, the family of models \eqref{eq:stochvol_genSDE} is able to reproduce the smile and skew effects in implied volatility structures (see Fouque et al.\ \cite{fouque2000} for an introduction to these concepts).

Under standard assumptions on the financial market, it is known that stochastic volatility models are incomplete and, accordingly, there exist infinitely many risk-neutral measures under which arbitrage-free pricing can be performed. Indeed, if the asset pays no dividends and the riskless (deterministic and time-dependent) interest rate is $r(t)$, then for any sufficiently regular deterministic function $\eta(t,x,v)$ the formula
\begin{equation} \label{eq:stochvol_mart_approach}
f^{(\eta)}(t,x,v) = e^{-\int_t^T r(u)\, du}\, \mathbb{E}_{\mathbb{Q}^{(\eta)}} \bigl[ Y \bigm| S_t = x, V_t = v \bigr]
\end{equation}
allows us to compute the arbitrage-free price at time $t$ of all contingent claims $Y$ with maturity at future time $T$. Here $\mathbb{Q}^{(\eta)}$ is a risk-neutral measure equivalent to $\mathbb{P}$ under which the dynamics of the process $(S,V)$ are given by
\begin{equation} \label{eq:stochvol_QetaSDE}
\begin{aligned}
d S_t & = r(t) S_t\, dt + g(V_t) S_t\, d\widehat{W}_t^1 \\[-0.52mm]
d V_t & = \bigl[a(t,V_t) - b(t,V_t) \Lambda_t \bigr]\, dt + b(t,V_t)\, d \widehat{W}_t^2
\end{aligned}
\end{equation}
where $\Lambda_t \equiv \Lambda(t,S_t,V_t)$ is defined as
\[
\Lambda(t,S_t,V_t) := \rho\, {\mu(t,S_t) - r(t) \over g(V_t)} + \sqrt{1-\rho^2}\, \eta(t,S_t,V_t)
\]
and $\widehat{W}_t^1$, $\widehat{W}_t^2$ are $\mathbb{Q}^{(\eta)}$-Brownian motions with correlation $\rho$.

The process $\eta(t,S_t,V_t)$, which is the so-called \emph{market price of volatility risk}, cannot be identified within the stochastic volatility model, so it must be exogenously specified. Unfortunately, there is no easy criterion for picking the right functional form for $\eta$; consequently, a common practice is to judiciously choose $\eta$ such that the resulting pricing problem is analytically tractable (see Section 2.7 of Fouque et al.\ \cite{fouque2000} and Section 10.9 of Lipton \cite{lipton2001}).

A DOC option is just a contingent claim $Y = (S_T - K)^+\, \mathds{1}_{\{\tau_H > T\}}$ where $\tau_H := \inf\{u \geq t: S_u \leq H \}$. This means that we can use Equation \eqref{eq:stochvol_mart_approach} for the pricing of DOC (and other barrier) options under the stochastic volatility model \eqref{eq:stochvol_QetaSDE} --- this is known as the \emph{martingale approach} to the pricing problem. Moreover, the barrier option price $f^{(\eta)}(t,x,v)$ can also be computed through a \emph{PDE approach}: by virtue of the Feynman-Kac theorem for Cauchy-Dirichlet problems for parabolic PDEs (Theorem \ref{thm:feynmankac_cdirich_rubio} in \ref{chap:ap_feynmankac}), $f^{(\eta)}$ is a solution of the two-space-dimensional terminal and boundary value problem
\begin{equation}
\begin{aligned}
\!\biggl({\partial \over \partial t} + \mathcal{L}^{(\eta)}\!\biggr) f^{(\eta)}(t,x,v) = 0, && \;\;\;\; t \in [0,T],\; x > H \\[-0.5mm]
f^{(\eta)}(T,x,v) = (x - K)^+, && x > H \\
f^{(\eta)}(t,H,v) = 0, && t \in [0,T]
\end{aligned}
\end{equation}
where
\begin{align*}
\mathcal{L}^{(\eta)} = \, & {1 \over 2} g^2(v) x^2 {\partial^2 \over \partial x^2} + \rho\, b(t,v) x g(v) {\partial^2 \over \partial x \partial v} + {1 \over 2} b^2(t,v) {\partial^2 \over \partial v^2} \\
& \!\!\! + r(t) x {\partial \over \partial x} + \bigl[ a(t,v) - b(t,v) \Lambda(t,x,v) \bigr] {\partial \over \partial v} - r(t)\, \textrm{Id}.
\end{align*}

The adaptation of these two pricing approaches to time-dependent barrier options is simple: we just need to redefine the stopping time as $\tau_H \!:=\! \inf\{u \geq t: S_u \leq H(u) \}$ and to replace the boundary condition of the PDE problem by $f^{(\eta)}(t,H(t),v) = 0$. It is also easy to generalize further to the case of options whose (down) barrier $H(t,v)$ depends both on the time and on the (random) volatility: the stopping time becomes $\tau_H := \inf\{u \geq t: S_u \leq H(u,V_u) \}$ and the boundary condition becomes $f^{(\eta)\!}(t,H(t,v),v) \!=\! 0$. We will be dealing with this more general class of barrier options in the next section because time and volatility-dependent barrier options will turn out to be very useful for the derivation of an approximate pricing formula for options with constant barriers.

\section{An asymptotic expansion approach to barrier option pricing} \label{chap:main}

In this section we will tackle the problem of pricing barrier options under the 2-hypergeometric stochastic volatility --- a particular case of the $\alpha$-hypergeometric stochastic volatility model which was defined by Da Fonseca and Martini \cite{fonseca2015} as follows:

\begin{definition}
The \emph{$\alpha$-hypergeometric stochastic volatility model} is the Markovian diffusion model with dynamics
\begin{equation} \label{eq:hypergeo_def_dyn}
\begin{aligned}
dS_t & = r(t) S_t dt + e^{V_t} S_t dW_t^1\\
dV_t & = \left(a-{c \over 2} e^{\alpha V_t}\right) dt + \theta\, dW_t^2
\end{aligned}
\end{equation}
where $W^1$ and $W^2$ are Brownian motions with correlation $\rho$, and $c, \alpha, \theta > 0$, $a \in \mathbb{R}$ are constants.
\end{definition}

Like Da Fonseca and Martini \cite{fonseca2015}, we assume that the model is given directly under a risk-neutral measure $\mathbb{Q}$. The deterministic function $r(t)$ represents the (possibly time-dependent) interest rate, while the parameters $a$ and $c$ can be used to set the market price of volatility risk.

It is important to emphasize that the formulation of the $\alpha$-hypergeometric stochastic volatility model given by Da Fonseca and Martini \cite{fonseca2015} and by Privault and She \cite{privaultshe2016} does not include the drift term $r(t) S_t\, dt$. If, as in these two papers, the goal is to price vanilla options, then such a zero interest rate assumption does not entail any loss of generality because the general case of a nonzero interest rate can be reduced to the case $r(t)=0$ by rewriting the pricing equation \textit{in forward terms} (cf.\ e.g.\ Subsection 9.2.1 of Lipton \cite{lipton2001}). However, this argument breaks down when dealing with barrier options, so the model with nonzero drift must be considered for our barrier option pricing problem.

So as to lighten the notation, we will henceforth assume that the interest rate is constant, i.e, $r(t) \equiv r$ and therefore $\int_t^u r(s)\, ds = r(u-t)$. We will also assume that the asset pays no dividends; as usual, the extension to assets with a continuously paid deterministic dividend is straightforward.

\subsection{The small vol of vol expansion} \label{sec:expansion}

Our approach to the barrier option pricing problem is based on a PDE regular perturbation method --- known as the \emph{small vol of vol asymptotic expansion} --- which consists in rewriting the model as a perturbed Black and Scholes model so as to derive a series expansion of the exact stochastic volatility price around the Black and Scholes price, which should converge when the perturbation parameter tends to zero. Our first step is thus to take the 2-hypergeometric model \eqref{eq:hypergeo_def_dyn} and replace the constant $\theta$ by a small parameter $\eps > 0$:
\begin{equation} \label{eq:2hyp_dyn_drift_smallvol}
\begin{aligned}
dS_t^\eps & = r S_t^\eps dt + e^{V_t^\eps} S_t^\eps dW_t^1 \\
dV_t^\eps & = \left(a-{c \over 2} e^{2 V_t^\eps}\right) dt + \eps \theta dW_t^2.
\end{aligned}
\end{equation}
We will assume that $a>0$ so as to assure that the log-volatility process is mean-reverting.

It is worth pointing out that a somewhat more general approach consists in replacing $\theta$ by a generic function $\eps \psi(t,v)$. The more general case is handled in essentially the same way (cf.\ \cite{thesis}).

Let $\hat{h}(t,v)$ be a generic time and volatility-dependent barrier function, to be specified later. The (exact) price $\hat{f}^\eps(t,x,v)$ of the DOC option with barrier function $\hat{h}(t,v)$ under the model \eqref{eq:2hyp_dyn_drift_smallvol} is defined (in the PDE approach) as the solution of the terminal and boundary value problem
\begin{equation} \label{eq:pde_exactbarrierprice}
\begin{aligned}
\!\biggl({\partial \over \partial t} + \mathcal{L}^\eps\biggr) \hat{f}^\eps (t,x,v) = 0&, &\;\;\;\; t \in [0,T],\, x > \hat{h}(t,v)
\\
\hat{f}^\eps(T,x,v)  = (x-K)^+&, & x > \hat{h}(T,v) \\
\hat{f}^\eps(t,\hat{h}(t,v),v) = 0&, & t \in [0,T]
\end{aligned}
\end{equation}
where 
\begin{equation}
\label{eq:PDEexpansion_L0L1L2} \begin{gathered} \mathcal{L}^\eps = \mathcal{L}_0 + \eps \mathcal{L}_1 + \eps^2 \mathcal{L}_2, \\[2pt]
\mathcal{L}_0 = \left(a-{c \over 2}e^{2v}\right)\!{\partial \over \partial v} + {x^2 \over 2} e^{2v}{\partial^2 \over \partial x^2} + rx {\partial \over \partial x} - r\, \textrm{Id}, \qquad \mathcal{L}_1 = \rho x e^v {\partial^2 \over \partial x \partial v}, \qquad \mathcal{L}_2 = {1 \over 2} {\partial^2 \over \partial v^2}.
\end{gathered}
\end{equation}

Let us now formally assume that the price $\hat{f}^\eps(t,x,v)$ can be asymptotically expanded as $\hat{f}^\eps = \hat{f}_0 + \eps \hat{f}_1$ $+ \eps^2 \hat{f}_2 + \ldots$. Substituting this expansion into the terminal and boundary value problem \eqref{eq:pde_exactbarrierprice} and equating the terms of order $\eps^0$, $\eps^1$, $\eps^2$, $\ldots$, we obtain the system of PDEs
\begin{equation} \label{eq:regpert_PDE_system}
\begin{gathered}
{\partial \hat{f}_0 \over \partial t} + \mathcal{L}_0 \hat{f}_0 = 0, \qquad {\partial \hat{f}_1 \over \partial t} + \mathcal{L}_0 \hat{f}_1 + \mathcal{L}_1 \hat{f}_0 = 0, \qquad {\partial \hat{f}_2 \over \partial t} + \mathcal{L}_0 \hat{f}_2 + \mathcal{L}_1 \hat{f}_1 + \mathcal{L}_2 \hat{f}_0 = 0, \qquad \ldots
\end{gathered}
\end{equation}
with terminal conditions $\hat{f}_0(T,x,v) = (x-K)^+$ and $\hat{f}_j(T,x,v) = 0$ for $j=1,2\ldots$, and with boundary conditions $\hat{f}_j(t,\hat{h}(t,v),v) = 0$ for $j=0,1,2,\ldots$.

We intend to derive the first-order approximation for the option price and to prove that (under suitable regularity conditions) it converges in the sense that 
\begin{equation} \label{eq:pde_asymptoticexp}
\hat{f}^\eps(t,x,v) = \hat{f}_0(t,x,v) + \eps \hat{f}_1(t,x,v) + \mathcal{O}(\eps^2)
\end{equation}
when $\eps$ goes to zero, uniformly with respect to $(t,x,v)$ on compact subsets of $[0,T] \times \mathbb{R}^+ \times \mathbb{R}$.

\subsection{The zero-order term} \label{sec:zeroorder}

The zero-order term $\hat{f}_0(t,x,v)$ is defined as the solution of the terminal and boundary value problem
\begin{equation} \label{eq:pde_zeroorder}
\begin{aligned}
\!\biggl({\partial \over \partial t} + \mathcal{L}_0\biggr) \hat{f}_0 (t,x,v) = 0&, &\;\;\;\; t \in [0,T],\, x > \hat{h}(t,v)
\\
\hat{f}_0(T,x,v)  = (x-K)^+ &, & x > \hat{h}(T,v) \\
\hat{f}_0(t,\hat{h}(t,v),v) = 0 &, & t \in [0,T]
\end{aligned}
\end{equation}
In other words, $\hat{f}_0$ is simply the option price corresponding to the limiting case $\eps = 0$. The equivalent definition of this option price under the martingale pricing framework is
\begin{equation} \label{eq:zeroord_nonzerodrift_mart1}
\hat{f}_0(t,x,v) = e^{-r(T-t)} \mathbb{E}\Bigl[ (S_T^{t,v} - K)^+\, \mathds{1}_{\{ \tau_{\hat{h}} \geq T \}} \Bigm| S_t^{t,v} = x \Bigr]
\end{equation}
where $\tau_{\hat{h}} = \inf\{u \geq t: S_u^{t,v} \leq \hat{h}(u,V_u^{t,v})\}$ and $\{(S_u^{t,v}, V_u^{t,v})\}_{u \in [t,T]}$ denotes the diffusion process whose dynamics are given by the noiseless limit $\eps = 0$ of the model \eqref{eq:2hyp_dyn_drift_smallvol}. The (degenerate) log-volatility process $V_u^{t,v}$ is therefore the deterministic function of time which solves the ordinary differential equation $d V_u^{t,v} = (a - {c \over 2} e^{2V_u^{t,v}})\, du$ with initial condition $V_t^{t,v} = v$; the explicit solution is
\begin{equation} \label{eq:volatilityproc_degen}
V_u^{t,v} = v + a(u-t) - {1 \over 2} \log\biggl(\! 1+{c \over 2a} e^{2v}(e^{2a (u-t)} - 1)\!\biggr).
\end{equation}
In turn, $\{S_u^{t,v}\}_{u \in [t,T]}$ is simply a geometric Brownian motion with constant drift $r$ and time-dependent deterministic volatility $e^{V_u^{t,v}}$. 

For a given (fixed) initial time $t = t'$ and initial log-volatility $v = v'$, by recalling the obvious semigroup property $V_u^{t,V_t^{t',v'}}\!\! = V_u^{t',v'}$ ($t' \leq t \leq u$) we see that
\begin{equation} \label{eq:zeroord_nonzerodrift_mart2}
\hat{f}_0 (t,x,V_t^{t',v'}\!) = e^{-r(T-t)} \mathbb{E}\Bigl[ (S_T^{t',v'}\! - K)^+\, \mathds{1}_{\{ \tau_{\hat{h}} \geq T \}} \Bigm| S_t^{t',v'}\!\! = x \Bigr]
\end{equation}
where $\tau_{\hat{h}} = \inf\{u \geq t: S_u^{t',v'} \leq \hat{h}(u,V_u^{t',v'})\}$. The function $\hat{f}_0(t,x,V_t^{t',v'})$, which only depends on the variables $t$ and $x$, is clearly the definition of the price of a DOC option under a Black and Scholes model where the interest rate is $r$, the time-dependent deterministic volatility is $e^{V_u^{t',v'}}\!$, $u \in [t',T]$ and the time-dependent barrier function is $\hat{H}(u) \equiv \hat{h}(u,V_{u}^{t',v'})$, $u \in [t',T]$.

The barrier option pricing problem under the Black and Scholes model $dS_t = \mu(t) S_{t\,} dt + \sigma(t) S_{t\,} dW_t$ has been studied in the literature. Rapisarda \cite{rapisarda2005} and Dorfleitner et al.\ \cite{dorfleitner2008} showed that the conditional expectation \eqref{eq:zeroord_nonzerodrift_mart2} can be written in closed form provided the barrier function is of the form
\[
\hat{H}(u) = H_1 \exp\left\{ -\int_u^T \biggl( \mu(s) - {1 + 2\beta \over 2} \sigma^2(s)\! \biggr) ds \right\}
\]
for $u \in [t',T]$, where $\beta \in \mathbb{R}$ and $H_1 > 0$ are parameters. In our case, this reduces to
\begin{equation} \label{eq:nonzerodrift_approxbar_beta}
\hat{H}(u) = H_1 \exp\left\{ - r(T-u) + {1 + 2\beta \over 2} \gamma^2(u,T,V_u^{t',v'}) \right\}
\end{equation}
where
\[
\gamma^2(t,u,v) :=\! \int_t^u e^{2V_s^{t,v}} ds\, = {1 \over c} \log\Bigl( 1 + {c \over 2a} e^{2v}(e^{2a(u-t)} - 1 )\!\Bigr).
\]
Unfortunately, to the best of our knowledge, an explicit expression for $\hat{f}_0(t,x,V_t^{t',v'})$ cannot be obtained unless $\hat{H}(u)$ has this particular functional form. For this reason, until Subsection \ref{sec:conv_proof} we will assume that the barrier function $\hat{h}(t,v)$ takes the specific form
\begin{equation} \label{eq:nonzerodrift_firstbarrierextbeta}
\hat{h}(t,v) = H_1 \exp\left\{ -r(T-t) + {1 + 2\beta \over 2} \gamma^2(t,T,v) \right\}
\end{equation}
in the domain $(t,v) \in [0,T] \times \mathbb{R}$. Given this choice of barrier function, Equation (27) of Rapisarda \cite{rapisarda2005} yields the following result (which can also be directly deduced from the joint law given in \ref{chap:ap_jointlaw}):
\begin{proposition} 
Let $\hat{f}_0(t,x,v)$ be the zero-order term in the expansion \eqref{eq:pde_asymptoticexp}. Then
\begin{equation} \label{eq:zeroord_approx_closed_beta_aw}
\begin{aligned}
\hat{f}_0(t,x,v) = & \, x\, \mathcal{N}(d_1(t,x,v)) - K e^{-r(T-t)} \mathcal{N}(d_2(t,x,v)) \\
& - \biggl({\hat{h}(t,v) \over x}\biggr)^{\!2+2\beta} x\, \mathcal{N}(d_3(t,x,v)) \\
& + \biggl({\hat{h}(t,v) \over x}\biggr)^{\!2\beta} K e^{-r(T-t)} \mathcal{N}(d_4(t,x,v))\\[-7pt]
\end{aligned}
\end{equation}
for $t \in [0,T]$ and $x > \hat{h}(t,v)$, where
\begin{gather*}
d_1(t,x,v) = {1 \over \gamma(t,T,v)} \biggl( \log\Bigl({ x \over K\vee H_1}\Bigr) + r(T-t) + {1\over 2} \gamma^2(t,T,v) \!\biggr), \\[3pt]
d_2(t,x,v) = d_1(t,x,v) - \gamma(t,T,v), \\
d_3(t,x,v) = d_1(t,x,v) + {2 \over \gamma(t,T,v)} \log\biggl({\hat{h}(t,v) \over x}\biggr), \\
d_4(t,x,v) = d_2(t,x,v) + {2 \over \gamma(t,T,v)} \log\biggl({\hat{h}(t,v) \over x}\biggr)
\end{gather*}
and $\mathcal{N}(\cdot)$ is the standard normal cumulative distribution function.
\end{proposition}%

If we take the limit $H_1 \to 0$, then the barrier function converges pointwise to zero; consequently, the zero-order term \eqref{eq:zeroord_approx_closed_beta_aw} converges to 
\[
x \mathcal{N}(d_1(t,x,v))- K e^{-r(T-t)} \mathcal{N}(d_2(t,x,v)),
\]
which is precisely the zero-order term for the vanilla option price expansion of Privault and She \cite{privaultshe2016}.

\subsection{The first-order term} \label{sec:firstorder}

The first-order term solves
\begin{equation} \label{eq:pde_firstorder_preFK}
\begin{aligned}
\!\biggl({\partial \over \partial t} + \mathcal{L}_0\biggr) \hat{f}_1 = -\mathcal{L}_1 \hat{f}_0, && \;\;\;\; t \in [0,T],\; x > \hat{h}(t,v) \\
\hat{f}_1(T,x,v) = 0, &&  x > \hat{h}(T,v) \\
\hat{f}_1(t,\hat{h}(t,v),v) = 0, && \qquad t \in [0,T]
\end{aligned}
\end{equation}
where the operators $\mathcal{L}_0$ and $\mathcal{L}_1$ were defined in \eqref{eq:PDEexpansion_L0L1L2}.

The first step towards the computation of an explicit expression for the first order term is to give a stochastic representation formula for the solution of this terminal and boundary value problem:

\begin{lemma} \label{lem:fk_firstorder_1}
Assume that $K \geq H_1$. Then the function
\begin{equation} \label{eq:fk_nonzero_firstorder}
\widetilde{f}_1(t, x, v) = \mathbb{E}\biggl[ \int_{t}^{T \wedge \tau_{\hat{h}}\!} e^{-r(u-t)} \mathcal{L}_1 \hat{f}_0(u,S_u^{t,v}, V_u^{t,v}) \, du \biggm| S_t^{t,v}\! = x \biggr]
\end{equation}
(where the process $(S^{t,v}, V^{t,v})$ and the stopping time $\tau_{\hat{h}}$ are defined as in \eqref{eq:zeroord_nonzerodrift_mart1}) is the unique classical solution of the terminal and boundary value problem \eqref{eq:pde_firstorder_preFK}.
\end{lemma}

\begin{proof}
See \ref{chap:ap_prooflemma}.
\end{proof}

Notice that Lemma \ref{lem:fk_firstorder_1} in particular assures the existence and uniqueness of a classical solution to \eqref{eq:pde_firstorder_preFK}. This is a nontrivial issue \cite{heathschweizer2000, ghany2014,ghany2013} which is often ignored in the mathematical finance literature.

In the limit $H_1 \to 0$, the dominated convergence theorem assures that \eqref{eq:fk_nonzero_firstorder} converges to 
\[
\mathbb{E}\biggl[ \int_{t}^T e^{-r(u-t)} \mathcal{L}_1 \hat{f}_0(u,S_u^{t,v}, V_u^{t,v}) \, du \biggm| S_t^{t,v}\! = x \biggr].
\]
Unsurprisingly, this is the definition of the first-order term for vanilla options, cf.\ Privault and She \cite{privaultshe2016}.

\begin{table*}[b!]
\centering
\caption{Parameters in equations \eqref{eq:partialf0_compact} and \eqref{eq:firstorder_closed}. We take $A := 1 - {K \over K \vee H_1}$ and omit the arguments of the functions $\gamma(u,T,V_u^{t,v})$ and $\hat{H}(u)$.}
\def\arraystretch{1.45}
\begin{tabular}{c}
\begin{tabular}{cccc}
\hline
$j$ & $a_j$ & $\eta_j$ & $b_{j,0}$ \\
\hline
$1$ & $0$ & $1$ & ${\partial \gamma \over \partial v} \Bigl( 1 - {A \over \gamma^2}\Bigr)$ \\
$2$ & $(1+2\beta) {\partial \hat{H}^{2+2\beta} \over \partial v}$ & $-(1+2\beta)$ & $- A\, {\partial \gamma \over \partial v} {\hat{H}^{2+2\beta} \over \gamma^2} + (1+2\beta)\, \hat{H}^{2+2\beta} \Bigl( {\partial \gamma \over \partial v} + {2 \over \gamma} {\partial \log\hat{H} \over \partial v} \Bigr) + {A \over \gamma} {\partial \hat{H}^{2+2\beta} \over \partial v}$ \\
$3$ & $-2\beta K e^{-r(T-u)} {\partial \hat{H}^{2\beta} \over \partial v}$ & $-2\beta$ & $2\beta K e^{-r(T-u)} \hat{H}^{2\beta} \Bigl( {\partial \gamma \over \partial v} - {2 \over \gamma} {\partial \log\hat{H} \over \partial v} \Bigr)$ \\
\hline
\end{tabular} \\
\vspace{-3mm} \\
\begin{tabular}{ccccc}
\hline
$j$ & $b_{j,1}$ & $b_{j,2}$ & $\nu_j$ & $\kappa_j$ \\
\hline
$1$ & $-{\partial \gamma \over \partial v} {1 + A \over \gamma}$ & ${\partial \gamma \over \partial v} {A \over \gamma^2}$ & ${1 \over \gamma}$ & ${1 \over \gamma}\bigl[-\log(K \vee H_1) +r(T-u) + {\gamma^2 \over 2}\bigr]$\!\!\! \\
$2$ & \!\!$- {\hat{H}^{2+2\beta} \over \gamma} \Bigl\{\! A \Bigl( {\partial \gamma \over \partial v} + {2 \over \gamma} {\partial \log\hat{H} \over \partial v} \Bigr) +$\! {\small $(1+2\beta)$}\!\! ${\partial \gamma \over \partial v} \!\Bigr\}$ & $\hat{H}^{2+2\beta} {\partial \gamma \over \partial v} {A \over \gamma^2}$ & $-{1 \over \gamma}$ & ${1 \over \gamma}\bigl[\log\bigl({\hat{H}^2 \over K \vee H_1}\bigr) + r(T-u) + {\gamma^2 \over 2}\bigr]$ \\
$3$ & $2\beta K e^{-r(T-u)} \hat{H}^{2\beta} {\partial \gamma \over \partial v} {1 \over \gamma}$ & $0$ & $-{1 \over \gamma}$ & ${1 \over \gamma}\bigl[\log\bigl({\hat{H}^2 \over K \vee H_1}\bigr) + r(T-u) - {\gamma^2 \over 2}\bigr]$ \\
\hline
\end{tabular}
\end{tabular}
\label{tab:params_partialf0}
\end{table*}

The task is to derive an explicit form for the expected value \eqref{eq:fk_nonzero_firstorder}, which we may rewrite as
\begin{equation} \label{eq:fk_firstorder_rewrit}
\hat{f}_1(t, x, v) = \int_{t}^{T} e^{-r(u-t)} \rho e^{V_u^{t,v}} \int_{\hat{h}(u,V_u^{t,v})\!}^\infty w\, {\partial^2 \hat{f}_0 \over \partial x \partial v}(u,w, V_u^{t,v})\, \mathbb{Q}\Bigl[ S_u^{t,v} \in dw, \tau_{\hat{h}} > u \Bigm| S_t^{t,v}\! = x \Bigr] du.
\end{equation}
We claim that the joint law of $(S_u^{t,v}, \tau_{\hat{h}})$ is given by
\begin{align*} 
\nonumber \mathbb{Q}\Bigl[ & S_u^{t,v} \in dw,  \tau_{\hat{h}} > u \Bigm| S_t^{t,v}\! = x \Bigr] = \\
& \quad = {\mathds{1}_{\{w > \hat{H}(u)\}} \over \gamma(t,u,v)\, w} \Biggl[ n\biggl( {1 \over \gamma(t,u,v)} \bigl( \log w - \mu_1 \bigr) \biggr) - \Bigl({\hat{H}(t) \over x}\Bigr)^{2\beta} n\biggl( {1 \over \gamma(t,u,v)} \bigl( \log w - \mu_2 \bigr) \biggr) \Biggr]\, dw.
\end{align*}
where $\hat{H}(u) \equiv \hat{h}(u,V_u^{t,v})$, $\mu_i := \log x_i + r(u-t) - {1 \over 2} \gamma^2(t,u,v)$, $x_1 := x$ and $x_2 := {\hat{H}^2(t) \over x}$. See \ref{chap:ap_jointlaw} for the proof of this claim. Consequently, the inside integral in \eqref{eq:fk_firstorder_rewrit} equals
\begin{equation} \label{eq:f1_insideintegral}
\mathbb{E}\Bigl[ e^{W_1} {\partial^2 \hat{f}_0 \over \partial x \partial v}(u, e^{W_1}, V_u^{t,v})\, \mathds{1}_{\{W_1 > \log \hat{H}(u)\}} \Bigr] - \Bigl({\hat{H}(t) \over x}\Bigr)^{2\beta} \mathbb{E}\Bigl[ e^{W_2} {\partial^2 \hat{f}_0 \over \partial x \partial v}(u, e^{W_2}, V_u^{t,v})\, \mathds{1}_{\{W_2 > \log \hat{H}(u)\}} \Bigr]
\end{equation}
where $W_i \sim \text{Normal}\bigl(\mu_i, \gamma^2(t,u,v)\bigr)$. Now, by differentiation of \eqref{eq:zeroord_approx_closed_beta_aw} we have
\begin{equation} \label{eq:partialf0_compact}
e^W {\partial^2 \hat{f}_0 \over \partial x \partial v}(u,e^W,V_u^{t,v}) = \sum_{j=1}^3 \biggl[ a_j\, e^{\eta_{j\!}^{} W} \mathcal{N}\bigl(\nu_j W + \kappa_j \bigr) + \sum_{\ell=0}^2 b_{j,\ell} (\nu_j W + \kappa_j)^{\ell\,} e^{\eta_{j\!}^{} W} n\bigl(\nu_j W + \kappa_j\bigr) \biggr]
\end{equation}
where $a_j, \eta_j, b_{j,\ell}, \nu_j, \kappa_j$ are the functions given in Table \ref{tab:params_partialf0}.

If we substitute \eqref{eq:partialf0_compact} into \eqref{eq:f1_insideintegral}, we obtain a sum of expectations which can be analytically solved with the help of Lemmas \ref{lem:expect_closedform1} and \ref{lem:expect_closedform2} in \ref{chap:ap_expect_closedform}. We have thus proved that the first-order term admits the following explicit expression:

\begin{proposition} 
Let $\hat{f}_1(t,x,v)$ be the first-order term in the expansion \eqref{eq:pde_asymptoticexp}, and assume that $K \geq H_1$. Then
\begin{equation} \label{eq:firstorder_closed}
\begin{aligned}
& \hat{f}_1 (t, x, v) = \int_{t}^{T}\! e^{-r(u-t)} \rho\, e^{V_u^{t,v}} \times \\
& \quad\; \times \Biggl[ \sum_{j=1}^3 \biggl(a_j\, \Upsilon\bigl( \nu_j, \kappa_j, \eta_j; \mu_1, \gamma^2(t,u,v), \hat{L} \bigr) + \sum_{\ell=0}^2 b_{j,\ell}\, \Psi_\ell\bigl( \nu_j, \kappa_j, \eta_j; \mu_1, \gamma^2(t,u,v), \hat{L} \bigr) \biggr)\\
& \quad\;\;\;\; - \Bigl({\hat{H}(t) \over x}\Bigr)^{2\beta} \sum_{j=1}^3 \biggl(a_j\, \Upsilon\bigl( \nu_j, \kappa_j, \eta_j; \mu_2, \gamma^2(t,u,v), \hat{L} \bigr) + \sum_{\ell=0}^2 b_{j,\ell}\, \Psi_\ell\bigl( \nu_j, \kappa_j, \eta_j; \mu_2, \gamma^2(t,u,v), \hat{L} \bigr) \biggr) \Biggr] du
\end{aligned}
\end{equation}
for $t \in [0,T]$ and $x > \hat{h}(t,v)$. Here $\hat{L} \equiv \hat{L}(u) \equiv \log\hat{H}(u)$ with $\hat{H}(u) \equiv \hat{h}(u,V_u^{t,v})$;\, $\mu_i := \log x_i + r(u-t)$ $- {1 \over 2} \gamma^2(t,u,v)$ with $x_1 := x$ and $x_2 := {\hat{H}^2(t) \over x}$;\, $a_j, \eta_j, b_{j,\ell}, \nu_j, \eta_j$ are the functions of $u$ given in Table \ref{tab:params_partialf0};\, and $\Psi_\ell$, $\Upsilon$ are the functions defined respectively by equations \eqref{eq:expect_closedform1}, \eqref{eq:expect_closedform2} in \ref{chap:ap_expect_closedform}.
\end{proposition} \vspace{-2pt}

We observe that the numerical computation of the integral in \eqref{eq:firstorder_closed} is much easier than solving numerically the associated PDE problem \eqref{eq:pde_firstorder_preFK} or computing the expectation \eqref{eq:fk_nonzero_firstorder} via Monte Carlo simulation.

\subsection{Convergence of the asymptotic expansion} \label{sec:conv_proof}

Now that we derived an explicit expression for our first-order approximation \eqref{eq:pde_asymptoticexp}, it is time to demonstrate that it converges in the limit $\eps \to 0$. (The idea of the following proof is similar to that in Appendix B of Kato et al.\ \cite{kato2014}.)

Let us start by looking into the PDE problem which is satisfied by the remainder term of the first-order approximation. For $\eps > 0$, we define the remainder term as
\begin{equation} \label{eq:convproof_f2eps_def}
\hat{f}_2^\eps(t,x,v) := {1 \over \eps^2} \Bigl[ \hat{f}^{\eps}(t,x,v) - \Bigl(\hat{f}_0(t,x,v) + \eps \hat{f}_1(t,x,v)\Bigr) \Bigr].
\end{equation}
Then, $\hat{f}_2^\eps$ satisfies the terminal and boundary value problem
\begin{equation} \label{eq:convproof_pde1}
\begin{aligned}
\!\biggl({\partial \over \partial t} + \mathcal{L}^\eps\biggr) u = - g_2^\eps, && \;\;\;\; t \in [0,T],\, x > \hat{h}(t,v) \\[-2pt]
u(T,x,v) = 0, && x > \hat{h}(T,v) \\
u(t,\hat{h}(t,v),v) = 0, && t \in [0,T].
\end{aligned}
\end{equation}
where $\mathcal{L}^\eps$ is the partial differential operator from \eqref{eq:PDEexpansion_L0L1L2}, and the nonhomogeneity term is
\begin{equation} \label{eq:convproof_g2eps}
g_2^\eps(t,x,v) := \mathcal{L}_2 \hat{f}_0(t,x,v) + \bigl( \mathcal{L}_1 + \eps \mathcal{L}_2 \bigr) \hat{f}_1(t,x,v).
\end{equation}
This is easily seen to be true by recalling that the functions $\hat{f}^\eps$, $\hat{f}_0$ and $\hat{f}_1$ are the unique solutions of the terminal and boundary value problems \eqref{eq:pde_exactbarrierprice}, \eqref{eq:pde_zeroorder} and \eqref{eq:pde_firstorder_preFK}, respectively.

Next, we use a stochastic representation formula to define a candidate solution $\widetilde{f}_2^\eps$ for the PDE problem \eqref{eq:convproof_pde1}:
\begin{equation} \label{eq:convproof_f2stochrep}
\widetilde{f}_2^\eps(t,x,v) := \mathbb{E}\biggl[ \int_{t}^{T \wedge \tau_{\hat{h}}^{\eps}} e^{-r(u-t)} g_2^\eps(u,S_u^\eps, V_u^\eps)\, du \Bigm| S_{t}^\eps = x, V_{t}^\eps = v \Bigr]
\end{equation}
where $\tau_{\hat{h}}^{\eps} := \inf\{u \geq t: S_u^\eps \leq \hat{h}(u,V_u^\eps)\}$. We emphasize that the process $(S^\eps,V^\eps)$ in \eqref{eq:convproof_f2stochrep} follows the $2$-hypergeometric model \eqref{eq:2hyp_dyn_drift_smallvol} with $\eps > 0$; in particular, here $V_t^\eps$ is a nondeterministic process.

We intend to establish a growth estimate for our candidate solution $\widetilde{f}_2^\eps$. As a preliminary step, let us first obtain an upper bound for the growth of the function $g_2^\eps$ defined in \eqref{eq:convproof_g2eps}:

\begin{lemma} \label{lem:convproof_g2eps_est}
Assume that $K \geq H_1$. Then, the function $g_2^\eps$ satisfies the following growth condition: for any $\eps \geq 0$, there exist constants $C$, $k > 0$ such that
\[
|g_2^\eps(t,x,v)| \leq C \left(1 + |x|^{2k} + e^{2kv}\right)
\]
for all $t \in [0,T]$, $v \in \mathbb{R}$ and $x \geq \hat{h}(t,v)$.
\end{lemma}

\begin{proof}
We can obtain an explicit expression for the function $g_2^\eps$ by differentiating the expressions \eqref{eq:zeroord_approx_closed_beta_aw} and \eqref{eq:fk_firstorder_rewrit} of the zero and first-order terms respectively. After a tedious estimation procedure, the lemma follows. (See Appendix B in \cite{thesis}.)
\end{proof}

The next lemma provides the tool for transforming our growth estimate for $g_2^\eps$ into a growth estimate for the candidate solution $\widetilde{f}_2^\eps$:

\begin{lemma} \label{lem:convproof_momentest}
Let $(S^\eps, V^\eps)$ be the diffusion process with dynamics \eqref{eq:2hyp_dyn_drift_smallvol}. Then, for any $\eps \geq 0$, there exist constants $C$, $m > 0$ (which may depend on $k$) such that
\begin{align*}
\mathbb{E}\biggl[\sup_{t \leq u \leq T} \Bigl(\bigl|S_u^\eps\bigr|^{2k} + e^{2kV_u^\eps}\Bigr) \biggm| & \, S_t^\eps = x, V_t^\eps = v \biggr] \leq C \bigl(1 + |x|^{2m} + e^{2m v}\bigr)
\end{align*}
for all $t \in [0,T]$, $x > 0$ and $v \in \mathbb{R}$.
\end{lemma}

\begin{proof}
The estimate for $\sup_{t \leq u \leq T} e^{2kV_u^\eps}$ is obtained by using Itô's formula to derive the dynamics of $Z^\eps = e^{2V^\eps}$ and then estimating the moments of the process $Z^\eps$ through a comparison with a geometric Brownian motion. Then, the estimate for $\sup_{t \leq u \leq T} \bigl|S_u^\eps\bigr|^{2k}$ can be derived from the closed-form expression
\[
S_u^\eps = x \exp\biggl( r(u-t) - {1 \over 2} \int_t^u e^{2V_s^\eps} ds + \int_t^u e^{V_s^\eps} dW_s^1 \biggr).
\]
(The full proof is in \cite{thesis}, pp.\ 36-37).
\end{proof}

Let us now use the results from Lemmas \ref{lem:convproof_g2eps_est} and \ref{lem:convproof_momentest} to derive the desired upper bound on the growth of the function $\widetilde{f}_2^\eps$ defined in \eqref{eq:convproof_f2stochrep}: for any $\eps \geq 0$, there exist constants $C$, $m > 0$ which do not depend on $(t,x,v)$ such that
\begin{equation} \label{eq:convproof_f2eps_est}
\begin{aligned}
|\widetilde{f}_2^\eps(t,x,v)| & \leq \int_{t}^T\!  e^{-r(u-t)} \mathbb{E}\Bigl[ |g_2^\eps(u,S_u^\eps, V_u^\eps)|\, \mathds{1}_{\{S_u^\eps \geq \hat{h}(u,V_u^\eps)\}} \!\Bigm|\! S_{t}^\eps = x, V_{t}^\eps = v \Bigr] du \\[3pt]
& \leq C_1 \int_{t}^T\Bigl( 1 + \mathbb{E}\Bigl[ (S_u^\eps)^{2k} + e^{2k V_u^\eps} \Bigm| S_{t}^\eps = x, V_{t}^\eps = v \Bigr] \Bigr) du \\[2pt]
& \leq C \bigl( 1 + |x|^{2m} + e^{2m v} \bigr)
\end{aligned}
\end{equation}
for all $t \in [0,T]$, $x > 0$, $v \in \mathbb{R}$.

The only thing that remains to be proved is that the function $\widetilde{f}_2^\eps$, which we defined as a candidate solution for the PDE problem \eqref{eq:convproof_pde1}, is indeed its unique solution. In fact, if we prove this, then it will follow that $\widetilde{f}_2^\eps$ equals the remainder term $\hat{f}_2^\eps$ defined in \eqref{eq:convproof_f2eps_def}, and the estimate \eqref{eq:convproof_f2eps_est} will assure the convergence of the first-order expansion.

\begin{lemma}
Assume that $K \geq H_1$ and fix $\eps > 0$. Then, the function $\widetilde{f}_2^\eps(t,x,v)$ defined in \eqref{eq:convproof_f2stochrep} is the unique solution of the terminal and boundary value problem \eqref{eq:convproof_pde1}.
\end{lemma}

\begin{proof}
The key ingredient of the proof is to perform the change of variables $z=e^{2v}$ and $y= x - \hat{h}(t,v)$. It is then straightforward to show that the restated version of the problem is a consequence of the Feynman-Kac theorem for Cauchy-Dirichlet problems for parabolic PDEs. (See \cite{thesis}, pp.\ 38-41.)
\end{proof}

Summarizing, we have established the following convergence theorem:
\begin{theorem} 
Let $\hat{f}_0(t,x,v)$ and $\hat{f}_1(t,x,v)$ be, respectively, the zero and first-order term in the expansion \eqref{eq:pde_asymptoticexp} for the price $\hat{f}^\eps(t,x,v)$  of a DOC option with barrier function $\hat{h}(t,v)$ under the model \eqref{eq:2hyp_dyn_drift_smallvol}. Assume that $K \geq H_1$. Then, there exist positive constants $C$ and $m$ which are independent of $\eps \in [0,1]$ such that
\[
\Bigl| \hat{f}^\eps(t,x,v) - \Bigl( \hat{f}_0(t,x,v) + \eps \hat{f}_1(t,x,v) \Bigr) \Bigr| \leq C \left(1 + |x|^{2m} + e^{2mv} \right) \eps^2
\]
for all $t \in [0,T]$, $v \in \mathbb{R}$ and $x \geq \hat{h}(t,v)$.
\end{theorem}

\begin{remark}
To keep our analysis simple, the results of Subsections \ref{sec:firstorder} and \ref{sec:conv_proof} have been stated and proved for the regular DOC case, that is, for the case when $K \geq H_1$ and hence the option's payoff is continuous. The fact that the payoff of a (reverse) DOC with $K < H_1$ is discontinuous at the intersection of the terminal and boundary conditions introduces additional complications, but our results can be extended to this case through a regularization technique similar to that introduced by Papanicolaou et al.\ in \cite{papanicolaou2003}.

The technique relies on the fact that it is possible to choose suitable families of smooth payoffs $\phi_\delta(x)$ which approximate the discontinuous payoff $\phi(x) = (x-K)^+$ and converge pointwise to  $\phi(x)$ as $\delta \to 0$. The price $\hat{f}^{\eps, \delta}(t,x,v)$ of a barrier option with payoff $\phi_\delta(x)$ is defined by \eqref{eq:pde_exactbarrierprice} with  $\phi_\delta(x)$ substituting $(x-K)^+$ and, as in Subsection \ref{sec:expansion}, it can be asymptotically expanded as $\hat{f}^{\eps,\delta} = \hat{f}_0^\delta + \eps \hat{f}_1^\delta + \eps^2 \hat{f}_2^\delta + \ldots$. Using the same arguments from the regular DOC case, it is possible to derive a closed-form expression for the zero-order term $\hat{f}_0^\delta$ and the first-order term $\hat{f}_1^\delta$, and also to prove the convergence of the expansion for $\hat{f}^{\eps, \delta}(t,x,v)$. (In fact, for concreteness we have chosen the payoff $(x-K)^+$ in the terminal condition of \eqref{eq:pde_exactbarrierprice}; notwithstanding, the same techniques yield identical results when $(x-K)^+$ is replaced by any other continuous payoff. Hence the application of our method to the payoffs $\phi_\delta(x)$ provides approximations to the price of the reverse DOC option.)  The next step is to show that the results for the reverse DOC follow by taking the limit $\delta \to 0$; for instance, to obtain Lemma \ref{lem:fk_firstorder_1} one needs to prove that the solutions of $\bigl({\partial \over \partial t} + \mathcal{L}_0\bigr) \hat{f}_1^\delta = -\mathcal{L}_1 \hat{f}_0^\delta$ converge to the solution of \eqref{eq:pde_firstorder_preFK}, and that the stochastic representation formulas $\mathbb{E}\bigl[ \int_{t}^{T \wedge \tau_{\hat{h}}\!} e^{-r(u-t)} \mathcal{L}_1 \hat{f}_0^\delta(u,S_u^{t,v}, V_u^{t,v}) \, du \bigm| S_t^{t,v}\! = x \bigr]$ converge to \eqref{eq:fk_nonzero_firstorder}.

We point out that the family of approximating smooth payoffs can be defined so that each $\phi_\delta(x)$ is either an upper bound or a lower bound for $\phi(x)$ (so that the pointwise convergence as $\delta \to 0$ takes place either by above or by below, respectively). Each approximate reverse DOC option price $\hat{f}^{\eps,\delta}(t,x,v)$ then becomes, respectively, an upper bound or a lower bound for $\hat{f}^\eps(t,x,v)$. The difference between such upper and lower bounds is an estimate of the error which occurs when one approximates the reverse DOC payoff by a continuous payoff.

Further details are left for the interested reader.
\end{remark}

\subsection{Single and multi-stage approximations to constant barriers} \label{sec:multistage}

Recall that we have been assuming that the nonconstant barrier function is of the form \eqref{eq:nonzerodrift_firstbarrierextbeta}. For this reason, we have (in general) not been covering the case with greater practical interest, which is that of a barrier option with constant barrier $H$. Notwithstanding, an approximate pricing formula for an option with constant barrier can be obtained if the parameters of \eqref{eq:nonzerodrift_firstbarrierextbeta} are chosen in order that the time and volatility-dependent barrier function is as constant as possible.

Such choice of parameters should take into account the fact that our pricing strategy is based on a small vol of vol expansion which is performed around the noiseless limit $V_t^{t',v'}$ of the log-volatility process $V_t^\eps$. Therefore, if one wishes to compute the price of the option at time $t' \in [0,T]$ and the initial log-volatility is equal to $v'$, then the parameters $H_1$ and $\beta$ should be chosen such that $\hat{h}(t,V_t^{t',v'})$ is as close to the constant function $H$ as possible. The simplest choice is $H_1 = H$ and $\beta$ such that $\hat{h}(t',v') = H$, i.e.\ $\beta = {r(T-t') \over \gamma^2(t',T,v')} - {1 \over 2}$, but this choice can be improved by choosing the parameters in some optimal way (see e.g.\ page 3 of Rapisarda \cite{rapisarda2005}).

It should be noted that the two cases where the barrier function \eqref{eq:nonzerodrift_firstbarrierextbeta} can be chosen to be constant are the zero interest rate case (i.e, $r=0$) and the case where the initial volatility equals its invariant value (i.e, $v' = {1 \over 2} \log\bigl({2a \over c}\bigr)$). Otherwise, the choice $\beta = {r(T-t') \over \gamma^2(t',T,v')} - {1 \over 2}$ yields an approximation which is quite good for small maturities. For large maturities it is possible to improve the quality of the approximation through the multi-stage procedure which we describe next.

The idea of the multi-stage method is to resort to a stepwise procedure so as to generalize our pricing technique to the case of a piecewise-smooth barrier function which is of the form \eqref{eq:nonzerodrift_firstbarrierextbeta} in each subinterval of time. Specifically, inspired by the approach proposed in Section 3 of Dorfleitner et al.\ \cite{dorfleitner2008}, we now subdivide the interval $[t',T]$ into $n$ subintervals defined by $t' = T_0 < T_1 < \ldots < T_n = T$ and consider the continuous barrier function defined by
\begin{equation}
\hat{h}^{(n)} (t,v) := H_1 \exp\Bigl\{ -r(T-t) + \sum_{i=1}^n {1+2\beta_i \over 2}\, \mathds{1}_{\{t < T_i\}} \gamma^2\bigl(t \vee T_{i-1},T_i,V_{t \vee T_{i-1}}^{t, v}\bigr) \biggr\}
\label{eq:nonzerodrift_firstb_betamulti}
\end{equation}
which is piecewise of the form \eqref{eq:nonzerodrift_firstbarrierextbeta} in the sense that
\[
\hat{h}^{(n)}(t,v) = \hat{h}^{(n)}(T_i,V_{T_i}^{t,v}) \exp\left\{ -r(T_i-t) + {1+2\beta_i \over 2} \gamma^2(t,T_i,v) \right\}
\]
for $(t,v) \in [T_{i-1},T_i] \times \mathbb{R}$. Notice that if we set $\beta_i = \beta$ for all $i=1, \ldots, n$ we obtain \eqref{eq:nonzerodrift_firstbarrierextbeta}. But the idea here is to pick $\beta_1, \ldots, \beta_n$ so that $\hat{h}^{(n)}(t,v)$ is closer to $H$ than the single-stage barrier function $\hat{h}(t,v)$: our choice of $\beta_i$ should ensure that the barrier function is as constant as possible in the interval $[T_{i-1}, T_i]$. Much like in the single-stage approximation, the simplest choice is $H_1 = H$ and $\beta_i = {r(T_i-T_{i-1}) \over \gamma^2(T_{i-1},T_i,V_{T_{i-1}}^{t',v'})} - {1 \over 2}$.

In order to derive an explicit asymptotic pricing formula for the option with barrier function \eqref{eq:nonzerodrift_firstb_betamulti}, we take the exact price $\hat{f}^{(n)}(t,x,v)$, i.e.\ the solution of the PDE problem \eqref{eq:pde_exactbarrierprice} with $\hat{h}(t,v)$ replaced by $\hat{h}^{(n)}(t,v)$, and formally expand it as
\[
\hat{f}^{(n)}(t,x,v) = \hat{f}_0^{(n)}(t,x,v) + \eps \hat{f}_1^{(n)}(t,x,v) + \mathcal{O}(\eps^2)
\]
where the functions $\hat{f}_0^{(n)}$ and $\hat{f}_1^{(n)}$ satisfy \eqref{eq:regpert_PDE_system}. (Naturally, the nonconstant boundary conditions are now $\hat{f}_j^{(n)\!}(t,\hat{h}^{(n)}(t,v),v) = 0$ for $j=0,1$.)

The same argument from the single-stage framework shows that for our fixed initial time $t'$ and initial log-volatility $v'$, the zero-order term is again the price, under the same Black and Scholes model, of a DOC option with barrier $\hat{H}^{(n)}(t) \equiv \hat{h}^{(n)}(t,V_{t})$, i.e,
\begin{equation} \label{eq:zeroord_multistagedef}
\hat{f}_0^{(n)}(t,x) \equiv \hat{f}_0^{(n)}(t,x,V_t) = e^{-r(T-t)} \mathbb{E}\Bigl[ (S_T - K)^+\, \mathds{1}_{\{ m_{[t,T]} > 1 \}} \Bigm| S_t = x \Bigr]
\end{equation}
where $m_{[t_1,t_2]} := \min_{u \in [t_1,t_2]} {S_u \over \hat{H}^{(n)}(u)}$. (For simplicity, we are now writing $S_t$, $V_t$ instead of $S_t^{t',v'\!}$, $V_t^{t',v'\!}$.) The key observation here is that, by virtue of the tower property, the function defined in \eqref{eq:zeroord_multistagedef} satisfies
\[
\hat{f}_0^{(n)}(t,x) = e^{-r(u-t)} \mathbb{E}\Bigl[ \hat{f}_0^{(n)}(u,S_u)\, \mathds{1}_{\{m_{[t,u]} > 1\}} \Bigm| S_t = x \Bigr]
\]
for $t' \leq t \leq u \leq T$. But the barrier function $\hat{H}^{(n)}(t)$ is of the form \eqref{eq:nonzerodrift_approxbar_beta} in each subinterval $[T_{i-1},T_i]$; therefore, we can obtain an explicit expression for the zero-order term $\hat{f}_0^{(n)}(t',x)$ as follows:
\begin{enumerate}[leftmargin=5mm]
\item We represent $\hat{f}_0^{(n)}(T_{n-1},x)$ via the closed-form expression \eqref{eq:zeroord_approx_closed_beta_aw}, where $\hat{h}(t,v)$ becomes $\hat{H}^{(n)}(T_{n-1})$, $\beta$ is replaced by $\beta_n$ and $\gamma(t,T,v)$ is replaced by $\gamma(T_{n-1},T,V_{T_{n-1}})$.

\item For $i=n-2, \ldots, 0$, we explicitly write
\begin{equation} \label{eq:zeroord_multistage_iter}
\nonumber \hat{f}_0^{(n)}(T_i,x) = e^{-r(T_{i+1}-T_i)} \int_{\hat{H}^{(n)}(T_{i+1})}^\infty \hat{f}_0^{(n)}(T_{i+1},w)\, \mathbb{Q}\Bigl[ S_{T_{i+1}}\! \in dw,\, m_{[T_i,T_{i+1}]\!} < 1 \!\Bigm|\! S_{T_i} = x \Bigr]
\end{equation}
where, as shown in \ref{chap:ap_jointlaw},
\begin{equation} \label{eq:jointlaw_multistage}
\begin{aligned}
& \mathbb{Q}\Bigl[ S_{T_{i+1}} \in dw,  m_{[T_i,T_{i+1}]} < 1 \Bigm| S_{T_i} = x \Bigr] = \\
& \qquad\; = {\mathds{1}_{\{w > \hat{H}^{(n)}(T_{i+1})\}} \over \gamma(T_i,T_{i+1},V_{T_i})\, w} \Biggl[ n\biggl( {1 \over \gamma(T_i,T_{i+1},V_{T_i})} \bigl( \log w - \mu_1 \bigr) \biggr) \\
& \qquad\qquad\quad - \Bigl({\hat{H}^{(n)}(T_i) \over x}\Bigr)^{2\beta_{i+1}} n\biggl( {1 \over \gamma(T_i,T_{i+1},V_{T_i})} \bigl( \log w - \mu_2 \bigr) \biggr) \Biggr]\, dw
\end{aligned}
\end{equation}
with $\mu_i := \log x_i + r(T_{i+1}-T_i) - {1 \over 2} \gamma^2(T_i,T_{i+1},V_{T_i})$, $x_1 := x$ and $x_2 := {(\hat{H}^{(n)})^2(T_i) \over x}$.
\end{enumerate}

Eventually, we obtain a representation for $\hat{f}_0^{(n)}(t',x)$ as a multiple integral of an explicit function. Moreover, this integral representation formula for $\hat{f}_0^{(n)}(t',x)$ can be written in closed form in terms of the cumulative distribution function of the $n$-dimensional normal distribution --- see \ref{chap:ap_twostage}.

As for the first-order term, in analogy with Lemma \ref{lem:fk_firstorder_1}, we invoke the Feynman-Kac theorem and write it as
\[
\hat{f}_1^{(n)} (t, x) \equiv \hat{f}_1^{(n)}(t, x, V_t) = \int_{t}^{T} \mathbb{E}\biggl[ e^{-r(u-t)} \mathcal{L}_1 \hat{f}_0^{(n)}(u,S_u)\, \mathds{1}_{\{m_{[t,u]} > 1\}} \biggm| S_t\! = x \biggr]\, du.
\]
Note that the tower property now gives
\[
\hat{f}_1^{(n)}(t,x) = e^{-r(u-t)} \mathbb{E}\Bigl[ \hat{f}_1^{(n)}(u,S_u)\, \mathds{1}_{\{m_{[t,u]} > 1\}} \Bigm| S_t = x \Bigr] + \int_{t}^{u} \mathbb{E}\biggl[ e^{-r(\ell-t)} \mathcal{L}_1 \hat{f}_0^{(n)}(\ell,S_\ell)\, \mathds{1}_{\{m_{[t,\ell]} > 1\}} \biggm| S_t\! = x \biggr]\, d\ell
\]
for $t' \leq t \leq u \leq T$. We can obtain an explicit expression for $\hat{f}_1^{(n)}(t',x)$ through a stepwise procedure:
\begin{enumerate}[leftmargin=5mm]
\item $\hat{f}_0^{(n)}(T_{n-1},x)$ is obtained through the closed-form expression \eqref{eq:firstorder_closed}, where $\hat{h}(t,v)$ becomes $\hat{H}^{(n)}(T_{n-1})$, $\beta$ is replaced by $\beta_n$ and $\gamma(t,T,v)$ is replaced by $\gamma(T_{n-1},T,V_{T_{n-1}})$.

\item For $i=n-2, \ldots, 0$, we get
\begin{align*}
\hat{f}_1^{(n)}(T_i,x) & = \, e^{-r(T_{i+1}-T_i)} \int_{\hat{H}^{(n)}(T_{i+1})}^\infty \hat{f}_1^{(n)}(T_{i+1},w)\, \mathbb{Q}\Bigl[ S_{T_{i+1}} \in dw,\, m_{[T_i,T_{i+1}]} < 1 \Bigm| S_{T_i} = x \Bigr] \\[2pt]
& \; + \int_{T_i}^{T_{i+1}}\! e^{-r(u-T_i)} \int_{\hat{H}^{(n)}(u)}^\infty \mathcal{L}_1\hat{f}_0^{(n)}(u,w)\, \mathbb{Q}\Bigl[ S_u \in dw, m_{[T_i,u]} < 1 \Bigm| S_{T_i} = x \Bigr] du.
\end{align*}
where the joint laws are again of the form \eqref{eq:jointlaw_multistage}. The integrands are known from the previous steps, so this is an explicit integral representation formula which is amenable to numerical integration.
\end{enumerate}

It is worth pointing out that the justification of the validity of the Feynman-Kac theorem is somewhat more delicate in this multi-stage setting. We will not deal with the technicalities here, but we do note that the natural strategy to deal with the lack of global smoothness consists in applying the Feynman-Kac theorem sequentially in each interval $[T_{n-1},T], \ldots, [t',T_1]$.

As a final remark, let us mention that the choice of $n$ --- and in particular the choice between the single and the multi-stage methods --- should be a compromise between computational speed and numerical accuracy, depending on the practical problem at hand.

\section{Numerical examples} \label{chap:numerics}

To demonstrate the validity and the practical usefulness of the barrier option pricing technique proposed in this paper, we shall now compare the numerical values of the exact option price under the 2-hypergeometric model with the approximate prices obtained through the first-order approximation derived in Sections \ref{sec:zeroorder} and \ref{sec:firstorder}.

Table \ref{tab:benchmark_comparison_MCcomb} shows the values of the exact and approximate prices for various cases, corresponding to different combinations of the model parameters. The exact prices (in the ``Benchmark" column) were obtained via Monte Carlo simulation of the exact solution of the PDE problem \eqref{eq:pde_exactbarrierprice} with constant barrier $H$. The Brownian bridge technique was used to assure the unbiasedness of the estimator (cf.\ Section 1.1 of Gobet \cite{gobet2000}). Two different schemes were used for the discretization of the stochastic differential equation: the usual Euler-Maruyama discretization of the process $(\log S_t^\eps, V_t^\eps)$, and an alternative scheme based on the fact that, as shown in page 3 of Da Fonseca and Martini \cite{fonseca2015}, the explicit closed-form expression for the process $e^{2V_u^\eps}$ is given by
\[
e^{2V_u^\eps} = {e^{2v} A_u^\eps \over 1 + c\, e^{2v} \int_t^u A_s^\eps\, ds}, \qquad u \geq t,
\]
where $A_u^\eps = \exp\{2a(u-t) + 2\eps (W_u^2 - W_t^2) \}$ is a geometric Brownian motion. (In the alternative scheme, the Euler-Maruyama method is instead used to simulate the process $A_u^\eps$ and the discretized values of $V_u^\eps$ are then obtained in the natural way.) The results obtained through the two discretization schemes were verified to be consistent, the difference being less than two standard errors. Our benchmarks were calculated as an average of these values, and the corresponding Monte Carlo standard error is shown in parentheses. In turn, the first-order approximate solutions (in the ``$\hat{f}_0 + \eps \hat{f}_1$" column) were computed via the explicit expressions \eqref{eq:zeroord_approx_closed_beta_aw} and \eqref{eq:firstorder_closed}, with $H_1 = H$ and $\beta = {r(T-t) \over \gamma^2(t,T,v)} - {1 \over 2}$ as proposed in Subsection \ref{sec:multistage}. For comparative purposes, the associated zero-order approximation $\hat{f}_0$ is also shown, as well as the Black and Scholes barrier option price $f_{BS}$ with constant volatility $\sigma = e^v$. In the case $e^{2v} = 0.04$ the latter two coincide, as the log-volatility function \eqref{eq:volatilityproc_degen} and the barrier function \eqref{eq:nonzerodrift_approxbar_beta} become constant.

The results in Table \ref{tab:benchmark_comparison_MCcomb} indicate that the first-order asymptotic formula consistently yields accurate estimates of the true price of the DOC option under the 2-hypergeometric model. In particular, the first-order expansion correctly captures the fact that the barrier option price decreases when the parameter $\rho$, i.e.\ the correlation between the asset price and the volatility shocks, becomes more negative. The zero-order approximation, which is insensitive to the value of $\rho$, produces larger errors, especially when the asset and volatility processes have stronger negative correlation. As for the plain Black and Scholes pricing formula, it leads to substantial errors, namely when the value of the initial squared volatility $e^{2v}$ is ``unusual", i.e, differs significantly from its long-term mean value ${2a \over c} = 0.04$.

\begin{table*}[t!]
\centering
\small
\caption{Comparison between the approximate option prices and the Monte Carlo benchmarks with $10\,000\,000$ sample paths and $100\,000$ time discretization steps. The errors are relative errors with respect to the benchmark. ($\eps=0.1$; $c=10$; $a=0.2$; $r=0.01$; $K=104$; $T=1$; $t=0$; $x=100$.)}
\setlength{\tabcolsep}{5.5pt}
\def\arraystretch{1.45}
\begin{tabular}{cccp{0pt}cp{0pt}ccp{2pt}ccp{2pt}cc}
\hline
$\rho$ & $H$ & $e^{2v}$ && Benchmark && $\hat{f}_0 + \eps \hat{f}_1$ & Error && $\hat{f}_0$ & Error && $f_{BS}$ & Error \\
\hline
$-0.5$ & $90$ & $0.02$ && $4.2850$ {\scriptsize ($0.0026$)} && $4.2711$ & $-0.3247\%$ && $4.3272$ & $0.9852\%$ && $4.1220$ & $-3.8046\%$ \\
$-0.5$ & $90$ & $0.04$ && $5.5611$ {\scriptsize ($0.0037$)} && $5.5456$ & $-0.2781\%$ && $5.6098$ & $0.8758\%$ && $5.6098$ & $0.8758\%$  \\
$-0.5$ & $90$ & $0.08$ && $6.5967$ {\scriptsize ($0.0049$)} && $6.5956$ & $-0.016\%$ && $6.6539$ & $0.8676\%$ && $6.9259$ & $4.9902\%$ \\
$-0.5$ & $85$ & $0.02$ && $4.5671$ {\scriptsize ($0.0026$)} && $4.5502$ & $-0.3685\%$ && $4.5946$ & $0.6024\%$ && $4.3356$ & $-5.0674\%$ \\
$-0.5$ & $85$ & $0.04$ && $6.3506$ {\scriptsize ($0.0038$)} && $6.3391$ & $-0.1817\%$ && $6.4010$ & $0.7938\%$ && $6.4010$ & $0.7938\%$ \\
$-0.5$ & $85$ & $0.08$ && $8.0577$ {\scriptsize ($0.0052$)} && $8.0563$ & $-0.0179\%$ && $8.1268$ & $0.8565\%$ && $8.6135$ & $6.8973\%$ \\
$-0.7$ & $90$ & $0.02$ && $4.2604$ {\scriptsize ($0.0026$)} && $4.2486$ & $-0.276\%$ && $4.3272$ & $1.5684\%$ && $4.1220$ & $-3.2490\%$ \\
$-0.7$ & $90$ & $0.04$ && $5.5378$ {\scriptsize ($0.0036$)} && $5.5199$ & $-0.3223\%$ && $5.6098$ & $1.2998\%$ && $5.6098$ & $1.2998\%$ \\
$-0.7$ & $90$ & $0.08$ && $6.5799$ {\scriptsize ($0.0048$)} && $6.5723$ & $-0.1146\%$ && $6.6539$ & $1.1257\%$ && $6.9259$ & $5.2588\%$ \\
$-0.7$ & $85$ & $0.02$ && $4.5475$ {\scriptsize ($0.0026$)} && $4.5325$ & $-0.3302\%$ && $4.5946$ & $1.0347\%$ && $4.3356$ & $-4.6594\%$ \\
$-0.7$ & $85$ & $0.04$ && $6.3309$ {\scriptsize ($0.0037$)} && $6.3142$ & $-0.2632\%$ && $6.4010$ & $1.1068\%$ && $6.4010$ & $1.1068\%$ \\
$-0.7$ & $85$ & $0.08$ && $8.0341$ {\scriptsize ($0.0051$)} && $8.0281$ & $-0.0752\%$ && $8.1268$ & $1.1526\%$ && $8.6135$ & $7.2112\%$ \\
\hline
\end{tabular}
\label{tab:benchmark_comparison_MCcomb}
\end{table*}

The huge computational burden of the Monte Carlo algorithm used to obtain the benchmarks makes it infeasible for practical applications. In contrast, the evaluation of the first-order approximation takes less than half of a second when the Mathematica function \textit{NIntegrate} is used to compute the integral in \eqref{eq:firstorder_closed}. This computation time is also clearly lower to that of other numerical schemes such as finite elements \cite{zhang2014} or boundary elements \cite{guardasoni2016}. The first-order asymptotic expansion proposed in this paper therefore provides a fast way of obtaining barrier option prices which capture the common financial market phenomena of volatility randomness and mean reversion.

We note that there is room for further improving the performance of the first-order option prices. Indeed, the slight negative bias of the first-order approximations may be corrected through a more optimal choice of the parameter $\beta$ in the approximating barrier function \eqref{eq:nonzerodrift_firstbarrierextbeta}, or by switching to a suitable multi-stage barrier function.

\section{Conclusions} \label{chap:conclusions}

In this article we established an asymptotic pricing formula for barrier options under the 2-hypergeometric stochastic volatility model. Moreover, we showed that our asymptotic technique is not just formal, as it converges when the perturbation parameter tends to zero.

An important feature of our method is that our explicit pricing formula only requires the numerical evaluation of a definite integral whose integrand is known in closed form. This calculation is fast and suitable for practical uses, unlike the computationally intensive methods which are commonly used for numerically computing option prices under stochastic volatility.

Even though our barrier option pricing technique requires two approximation steps, our numerical examples indicate that the resulting error is quite small. We also proposed a multi-stage method which can be employed to improve the quality of the approximation. Our investigation therefore shows that the perturbation (asymptotic) methods, which were introduced in the option pricing literature by Fouque et al.\ (cf.\ \cite{fouque2000} and references therein), are a simple yet effective technique for the evaluation of barrier option prices under the 2-hypergeometric stochastic volatility model.

It would be interesting to investigate whether the price of other exotic options (including American-type options) under the 2-hypergeometric stochastic volatility model can also be computed via the small vol of vol expansion method proposed in this work. We leave this task for future research. 

\section*{Acknowledgments}

Most of this work was carried out while the first author was at Instituto Superior Técnico, Universidade de Lisboa. The third author was partly supported by Fundação para a Ciência e Tecnologia (FCT/MEC) through the project CEMAPRE - UID/MULTI/00491/2013.

\appendix

\let\apsection\section
\renewcommand\section[1]{
	\renewcommand{\thesection}{Appendix \Alph{section}}
	\apsection{#1}
	\renewcommand{\thesection}{\Alph{section}}
}

\section{The Feynman-Kac theorem for Cauchy-Dirichlet problems for parabolic PDEs} \label{chap:ap_feynmankac}

In this appendix we state a version of the Feynman-Kac theorem whose proof was given by Rubio in \cite{rubio2011}.

\begin{theorem} \label{thm:feynmankac_cdirich_rubio}
Let $D \subset \mathbb{R}^d$ be an open, connected and possibly unbounded set whose boundary $\partial D$ has the outside strong sphere property, and let $\lambda \in (0,1)$. Assume that:
\begin{enumerate}[itemsep=2.4pt]
\item[(i)] For all $n > 1$, the functions $\sigma^{ij}(t,x)$ and $b^i(t,x)$ are $\lambda$-H\"{o}lder continuous in $t$ and Lipschitz continuous in $x$ in the domain $\{(t,x): 0 \leq t \leq T, |x| \leq n\}$;

\item[(ii)] There exists $K_1$ such that
\[
\,\sum_{i,j=1\!}^d |\sigma^{ij}(t,x)|^2 + \sum_{i=1\!}^d |b^i(t,x)|^{2\!} \leq K_1\bigl( 1 + |x|^2 \bigr)
\]
for all $(t,x) \in [0,T] \times \mathbb{R}^d$;

\item[(iii)] Let $B \subset \overline{D}$ be any bounded, open, connected set. There exists $\theta(B) > 0$ such that
\begin{equation} \label{eq:feynmankac_rubio_ellipt}
\sum_{i,j=1}^d a^{ij}(t,x) \xi_i \xi_j \geq \theta(B)\, |\xi|^2
\end{equation}
for all $(t,x) \in [0,T] \times \overline{B}$, $\xi \in \mathbb{R}^d$, where $a(t,x) = \bigl[a^{ij}(t,x)\bigr]_{i,j=1, \ldots, n} := \sigma \sigma'(t,x)$;

\item[(iv)] The functions $c(t,x)$ and $g(t,x)$ are $\lambda$-H\"{o}lder continuous in $t$ and Lipschitz continuous in $x$ in the domain $\{(t,x): 0 \leq t \leq T, x \in \overline{D}, |x| \leq n\}$;

\item[(v)] There exists $c_0 \geq 0$ such that $c(t,x) \leq c_0$ for all $(t,x) \in [0,T] \times \overline{D}$;

\item[(vi)] There exist constants $K_2, k > 0$ such that $|g(t,x)| \leq K_2\bigl( 1 + |x|^k \bigr)$ for all $(t,x) \in [0,T] \times \overline{D}$;

\item[(vii)] The functions $\phi(x)$ and $\varphi(t,x)$ are continuous and satisfy the consistency condition $\phi(x) = \varphi(T,x)$, $x \in \partial D$;

\item[(viii)] There exist constants $K_3, k > 0$ such that $|\phi(x)| + |\varphi(t,x)| \leq K_3\bigl( 1 + |x|^k \bigr)$ for all $(t,x) \in [0,T] \times \overline{D}$.
\end{enumerate}

Then, the unique solution $u \in C([0,T] \times \overline{D})\, \cap$ $C_\textrm{\emph{loc}}^{1,2,\lambda}((0,T) \times D)$ of the Cauchy-Dirichlet problem
\begin{align*}
{\partial u \over \partial t} + {1 \over 2} \sum_{i,j=1}^d a^{ij}(t,x) {\partial^2 u \over \partial x_i \partial x_j} + \sum_{i=1}^d b^i(t,x) {\partial u \over \partial x_i} + c(t,x	) u & = g(t,x) \qquad\; (t,x) \in [0,T] \times D \\[-3pt]
u(T,x) & = \phi(x) \qquad\quad\, x \in D \\
u(t,x) & = \varphi(t,x) \qquad\; (t,x) \in [0,T] \times \partial D
\end{align*}
is given by
\begin{align*}
u(t,x) = & \, \mathbb{E}\Bigl[ e^{\int_t^{\tau} c(u,X_u)\, du} \varphi(\tau, X_{\tau}) \mathds{1}_{\{\tau < T\}} \Bigm| X_t = x\Bigr] \\[2pt]
& + \mathbb{E}\Bigl[ e^{\int_t^T c(u,X_u)\, du} \phi( X_T ) \mathds{1}_{\{\tau \geq T\}} \Bigm| X_t = x\Bigr] \\[1pt]
& - \mathbb{E}\biggl[ \int_t^{\tau} e^{\int_t^u c(s,X_s)\, ds} g(u,X_u)\, du \biggm| X_t = x \biggr].
\end{align*}
Here $\tau = \inf\{ u \geq t: X_u \notin D \}$, $X$ is the $d$-dimensional Markovian diffusion process with dynamics
\[ 
dX_t^i = b^i(t, X_t)\, dt + \sum_{j=1}^d \sigma^{ij} (t, X_t)\, dW_t^j, \quad\; i=1, \ldots, d
\]
and $C_\textrm{\emph{loc}}^{1,2,\lambda}((0,T) \times D)$ is the space of all functions such that they and all their derivatives up to the second order in $x$ and first order in $t$ are $\lambda$-H\"{o}lder continuous. Furthermore, $u$ satisfies the growth estimate
\[
\sup_{t \in [0,T]} |u(t,x)| \leq C(c_0,K_1,K_2,K_3,k) \bigl( 1 + |x|^k \bigr), \quad\; x \in \overline{D}.
\]
\end{theorem}

\section{Proof of Lemma \ref{lem:fk_firstorder_1}} \label{chap:ap_prooflemma}

Since the complete proof is quite lengthy, here we shall only present the main ideas of the proof of Lemma \ref{lem:fk_firstorder_1}. For further details, we refer to \cite{thesis}, pp.\ 31-34.

We begin by carrying out the change of variables $y=x-\hat{h}(t,v)$, which reduces the problem to that of showing that the function $\widetilde{f}_1^*(t,y,v) := \widetilde{f}_1(t,y+ \hat{h}(t,v),v)$ is the unique solution of the terminal and boundary value problem obtained by replacing the operator $\mathcal{L}_0$ by $\mathcal{L}_0^* := \mathcal{L}_0 - \Bigl({\partial \hat{h} \over \partial t}\!(t,v) + \bigl(a-{c \over 2}e^{2 v}\bigr) {\partial \hat{h} \over \partial v}\!(t,v)\Bigr) {\partial \over \partial x}$ and by changing the nonhomogeneity term accordingly. Notice that, as a result of this change of variables, the boundary of the problem becomes constant.

Since our problem does not satisfy neither the ellipticity assumption (iii) nor the growth restrictions (ii) and (vi) of Theorem \ref{thm:feynmankac_cdirich_rubio} of \ref{chap:ap_feynmankac}, the desired result does not follow directly from this theorem. Nevertheless, it turns out that our lemma can be obtained by performing some adaptations to the proof given in \cite{rubio2011}. Indeed, under the assumptions of Lemma \ref{lem:fk_firstorder_1}:
\begin{itemize}[itemsep=2.4pt]
\item The property in Remark 2.4 of \cite{rubio2011} can be deduced from the fact that, in our setting, the process $Y_u^{t,v} := S_u^{t,v} - \hat{h}(u,V_u^{t,v})$, interpreted as a one-dimensional diffusion, satisfies the ellipticity condition \eqref{eq:feynmankac_rubio_ellipt} in \ref{chap:ap_feynmankac}.

\item Properties (i), (ii) and (iii) in Proposition 2.5 of \cite{rubio2011} can be obtained from the closed-form expression
\[
Y_u^{t,v} = (y+\hat{h}(t,v)) \exp\biggl\{ r(u-t) - {1 \over 2}\gamma^2(t,u,v) + \int_t^u e^{V_s^{t,v}} dW_s^1 \biggr\} - \hat{h}(u,V_u^{t,v})
\]
(where $y = Y_t^{t,v}$). In addition, the inequality
\begin{equation} \label{eq:prooflemma_momentest}
\sup_{v \in [M_0,M_1]} \mathbb{E}\biggl[ \sup_{t \leq u \leq T} \bigl|Y_u^{t,v}\bigr|^{2r} \biggm| Y_{t}^{t,v} = y \biggr] \leq C(M_0,M_1,r)\left( 1 + |y|^{2r} \right)
\end{equation}
for $t \in [0,T]$, $y > 0$, $M_0 < M_1$ and $r\geq 1$, follows from Corollary 2.5.12 of Krylov \cite{krylov1980}. (Equation \eqref{eq:prooflemma_momentest} is a weaker version of property (iv) in Proposition 2.5 of \cite{rubio2011}.)

\item Equation \eqref{eq:prooflemma_momentest} provides a growth estimate for the moments of $Y_u^{t,v}$, and the nonhomogeneity term $g(t,y,v)$ satisfies
\[
\bigl|g(t,y,v)\bigr| \leq K_2^*(M_0,M_1) \bigl(1+|y|^k\bigr) \qquad \text{for all } t \in [0,T],\, v \in [M_0,M_1],\, y > 0
\]
(this estimate can be derived from the closed-form expression \eqref{eq:zeroord_approx_closed_beta_aw}). Straightforward modifications of the proofs of Lemma 3.3, Lemma 3.4 and Theorem 4.1 of \cite{rubio2011} yield that these results are also valid under our weaker growth estimates.

\item Theorem 5.4 of \cite{rubio2011} can be proved through a localization argument which relies on the fact that, for fixed $t$ and $v$, the degenerate log-volatility process \eqref{eq:volatilityproc_degen} is a bounded function of $u$ on the interval $[t,T]$.

\item The differentiability (cf.\ Section 4.2 of \cite{rubio2011}) of the stochastic representation formula $\widetilde{f}_1^*(t,y,v)$ can be proved by observing that $\widetilde{f}_1(t,x,v)$ admits the explicit expression given in the right hand side of \eqref{eq:firstorder_closed}, so that we may directly compute the derivatives and verify that they have the required continuity.
\end{itemize}
As in \cite{rubio2011}, the desired conclusion follows. \qed

\section{The joint law of a geometric Brownian motion with time-dependent volatility and the hitting time of a suitable barrier} \label{chap:ap_jointlaw}

Let $V(u)$ be some deterministic function defined for all $u \in [t',T]$, where $t'$ is the initial time and $T$ is the final time. Let $\gamma^2(u) = \int_{t'}^u e^{2V(\ell)} d\ell$.

Here we shall deduce the joint law of $(S_u,\tau)$, where $\{S_u\}_{u \in [t',T]}$ is a geometric Brownian motion with constant drift $r$ and time-dependent deterministic volatility $e^{V(u)}$, i.e,
\[
S_u = x \exp\biggl\{ r(u-t') - {1 \over 2} \gamma^2(u) + \int_{t'}^u e^{V(\ell)} dW_\ell \biggr\}, \quad\! u \in [t'\!,T]
\]
and $\tau = \inf\{ u \in[t',T]: S_u \leq \hat{H}(u) \}$ where
\[
\hat{H}(u) = H_1 \exp\left\{ -r(T-u) + {1 + 2\beta \over 2} \bigl(\gamma^2(T) - \gamma^2(u)\bigr) \right\}.
\]

Define
\[
Z_u := {S_u \over \hat{H}(u)} = z \exp\biggl\{\beta \gamma^2(u) + \int_{t'}^u e^{V(\ell)} dW_\ell \biggr\}, \quad\; u \in [t',T]
\]
where $z := {x \over \hat{H}(t')}$. Moreover, let
\[
A_s := Z_{\vartheta(s)} \qquad\quad \text{where} \;\; \vartheta(s) := \inf\{u \geq t': \gamma^2(u) \geq s\} \;\; \text{and} \;\; s \in [0,\gamma^2(T)].
\]
Since $\gamma^2(u)$ is a continuous and strictly increasing function, we have $\gamma^2(\vartheta(s)) = s$ and $\vartheta(\gamma^2(u)) = u$; therefore
\[
A_s = z \exp\bigl\{\beta s + \overline{W}_s \bigr\}, \quad\; s \in [0,\gamma^2(T)]
\]
where $\overline{W}_s := \int_{t'}^{\vartheta(s)} e^{V(\ell)} dW_\ell$ is (up to indistinguishability) a $\mathbb{Q}$-Brownian motion. The latter claim follows from the change of time theorem (e.g.\ Theorem 9.3 of Chung and Williams \cite{chungwilliams1990}); note that $\gamma^2(u)$ is the quadratic variation of the local martingale $\int_{t'}^u e^{V(\ell)} dW_\ell$.

Observe also that $\{\tau > u \} = \{\min_{\ell \in [t',u]} Z_\ell > 1 \} = \{\min_{s \in [0,\gamma^2(u)]} A_s > 1 \}$. So we can compute, for $u \in [t',T]$ and $c > \hat{H}(u)$,
\begin{align*}
\mathbb{Q}\bigl[ S_u > c, \tau > u \bigr] & = \mathbb{Q}\biggl[ A_{\gamma^2(u)} > {c \over \hat{H}(u)},\; \min_{s \in [0,\gamma^2(u)]} A_s > 1 \biggr] \\
& = \mathcal{N}\Biggl( {1 \over \gamma(u)} \biggl( \log\Bigl( {z \hat{H}(u) \over c} \Bigr) + \beta \gamma^2(u) \biggr) \Biggr) - z^{-2\beta} \mathcal{N}\Biggl( {1 \over \gamma(u)} \biggl( \log\Bigl( {\hat{H}(u) \over cz} \Bigr) + \beta \gamma^2(u) \biggr) \Biggr) \\
& = \mathcal{N}\biggl( -{1 \over \gamma(u)} \bigl( \log c - \mu_1 \bigr) \biggr) - \Bigl({\hat{H}(t') \over x}\Bigr)^{2\beta} \mathcal{N}\biggl( -{1 \over \gamma(u)} \bigl( \log c - \mu_2 \bigr) \biggr).
\end{align*}
where $\mu_i := \log x_i + r(u-t) - {1 \over 2} \gamma^2(u)$, $x_1 := x$ and $x_2 := {\hat{H}^2(t) \over x}$. In the second equality we have used the known law of a geometric Brownian motion and its running minimum, which can be found in Subsection 3.3.1 of Jeanblanc et al.\ \cite{jeanblanc2009}. Differentiating, we conclude that
\[ 
\mathbb{Q}\bigl[ S_u \in dc, \tau > u \bigr] = {\mathds{1}_{\{c > \hat{H}(u)\}} \over \gamma(u)\, c} \Biggl[ n\biggl( {1 \over \gamma(u)} \bigl( \log c - \mu_1 \bigr) \biggr) - \Bigl({\hat{H}(t') \over x}\Bigr)^{2\beta} n\biggl( {1 \over \gamma(u)} \bigl( \log c - \mu_2 \bigr) \biggr) \Biggr]\, dc.
\]
where $n(\cdot)$ is the standard normal probability density function.

\section{Closed-form expressions for the expectation of functions of a Normal random variable} \label{chap:ap_expect_closedform}

\begin{lemma} \label{lem:expect_closedform1}
Define, for $\ell=0,1,2$,
\[
\Psi_\ell(\nu, \kappa, \eta; \mu, \sigma^{2\!}, L) := \mathbb{E}\Bigl[ (\nu W + \kappa)^{\ell\,} e^{\eta W} n(\nu W + \kappa)\, \mathds{1}_{\{W>L\}} \Bigm| W \sim \emph{Normal}(\mu,\sigma^2) \Bigr].
\]
Then,
\begin{equation} \label{eq:expect_closedform1}
\begin{aligned}
\Psi_\ell(\nu, \kappa, \eta; \mu, \sigma^{2\!}, L) & = \zeta{s^2 \over \sigma} c_{1,\ell} \biggl[ s\, \mathcal{N}\biggl( {m - L \over s} \biggr) + (L-m)\, n\biggl( {m - L \over s} \biggr)\biggr] \\
& \;\; + \zeta{s^2\over \sigma} \bigl(c_{2,\ell} + 2c_{1,\ell} m\bigr)\, n\biggl( {m - L \over s} \biggr) + \zeta{s \over \sigma} \bigl(c_{3,\ell} + c_{2,\ell} m + c_{1,\ell} m^2\bigr)\, \mathcal{N}\biggl( {m - L \over s} \biggr)
\end{aligned}
\end{equation}
where
\begin{gather*}
\zeta = {1 \over \sqrt{2\pi}} \exp\Bigl\{ -{1 \over 2} \Bigl({\mu^2 \over \sigma^2} - {(\mu - \sigma^2 (\nu\kappa - \eta))^2 \over \sigma^2 (1 + \sigma^2 \nu^2)} + \kappa^2 \Bigl)\Bigr\},\\[2.5pt]
m = {\mu - \sigma^2 (\nu\kappa - \eta) \over 1 + \sigma^2 \nu^2}, \qquad s^2 = {\sigma^2 \over 1 + \sigma^2 \nu^2}, \\[2pt]
c_{1,0} = c_{1,1} = c_{2,0} = 0, \quad\;\; c_{1,2} = \nu^2, \quad\;\; c_{2,1} = \nu, \quad\;\; c_{2,2} = 2 \nu \kappa, \quad\;\; c_{3,0} = 1, \quad\;\; c_{3,1} = \kappa, \quad\;\; c_{3,2} = \kappa^2.
\end{gather*}
\end{lemma}

\begin{proof}
See Appendix A.1 of \cite{thesis}.
\end{proof} \vspace{4pt}

\begin{lemma} \label{lem:expect_closedform2}
Define
\[
\Upsilon(\nu, \kappa, \eta; \mu, \sigma^{2\!}, L) := \mathbb{E}\Bigl[ e^{\eta W} \mathcal{N}(\nu W + \kappa)\, \mathds{1}_{\{W>L\}} \Bigm| W \sim \emph{Normal}(\mu,\sigma^2) \Bigr].
\]
Then
\begin{equation} \label{eq:expect_closedform2}
\Upsilon(\nu, \kappa, \eta; \mu, \sigma^{2\!}, L) = \exp\biggl\{ {q_2^2 \over 4q_1} - q_3 \biggr\}\, \mathcal{N}_2\biggl( {q_2 \over \sqrt{2q_1}}, {2 q_1 q_5 - q_2 q_4 \over \sqrt{2q_1 (2q_1 + q_4^2)}}\, ; \, -{q_4 \over \sqrt{2q_1 + q_4^2}} \biggr)
\end{equation}
where
\[
\begin{gathered}
q_1 = {1 \over 2(1+	\sigma^2 \nu^2)}, \qquad q_2 = {\kappa + \nu(\mu+\sigma^2 \eta) \over 1+\sigma^2\nu^2},\\[3pt]
q_3 = {(\kappa + \mu \nu)^2 - 2\eta(\mu-\kappa \nu \sigma^2) - \eta^2 \sigma^2 \over 2(1+\sigma^2 \nu^2)},\\[3pt]
q_4 = {-\sigma \nu \over \sqrt{1 + \sigma^2 \nu^2}}, \quad\;\; q_5 = {\mu + \sigma^2 (\eta - \kappa \nu) - L(1+\sigma^2 \nu^2) \over \sqrt{\sigma^2 (1+\sigma^2 \nu^2)}}
\end{gathered}
\]
and $\mathcal{N}_2(\,\cdot\,, \cdot\,; \rho)$ denotes the cumulative distribution function of a bivariate normal random variable with zero means, unit variances and correlation $\rho$.
\end{lemma}

\begin{proof}

\begin{align*}
& \mathbb{E}\Bigl[ e^{\eta W} \mathcal{N}(\nu W + \kappa)\, \mathds{1}_{\{W>L\}} \Bigm| W \sim \textrm{Normal}(\mu,\sigma^2) \Bigr] \\
& \qquad\qquad = \int_{-\infty}^0 \mathbb{E}\Bigl[ e^{\eta W} n(\nu W + \kappa + z)\, \mathds{1}_{\{W>L\}} \Bigm| W \sim \textrm{Normal}(\mu,\sigma^2) \Bigr]\, dz \\
& \qquad\qquad = {1 \over \sqrt{2\pi(1+\sigma^2 \nu^2)}} \int_{-\infty}^0 \exp\{q_1 z^2 + q_2 z + q_3\}\, \mathcal{N}(q_5 z + q_6)\, dz \\
& \qquad\qquad = \exp\biggl\{ {q_2^2 \over 4q_1} - q_3 \biggr\}\, \mathcal{N}_2\biggl( {q_2 \over \sqrt{2q_1}}, {2 q_1 q_5 - q_2 q_4 \over \sqrt{2q_1 (2q_1 + q_4^2)}}\, ; \, -{q_4 \over \sqrt{2q_1 + q_4^2}} \biggr).
\end{align*} 
The second equality follows from Lemma \ref{lem:expect_closedform1}, and the last equality follows from Appendix A.3 in \cite{thesis}.
\end{proof}

\section{A closed-form two-stage formula for the zero-order term} \label{chap:ap_twostage}

In order to illustrate that the integrals in the representation formula \eqref{eq:zeroord_multistage_iter} for $\hat{f}_0^{(n)}(t',x) \equiv \hat{f}_0^{(n)}(t',x,v')$ can be computed in closed form, here we focus on the case $n=2$, where \eqref{eq:zeroord_multistage_iter} gives
\[
\begin{aligned}
& \hat{f}_0^{(2)}(t',x) = e^{-r(T_1-t')} \int_{\hat{H}^{(2)}(T_1)}^\infty {1 \over \gamma(t',T_1,v')\, w} \times \\
& \qquad\quad \times \Biggl[ n\biggl( {1 \over \gamma(t',T_1,v')} \bigl( \log w - \mu_1 \bigr) \biggr) - \Bigl({\hat{H}^{(2)}(t') \over x}\Bigr)^{2\beta_1} n\biggl( {1 \over \gamma(t',T_1,v')} \bigl( \log w - \mu_2 \bigr) \biggr) \Biggr] \times \\
& \qquad\quad \times \Biggl[ w\, \mathcal{N}(d_1(T_1,w,V_{T_1})) - K e^{-r(T-T_1)} \mathcal{N}(d_2(T_1,w,V_{T_1})) \\
& \qquad\quad\quad\;\; - \biggl({\hat{H}^{(2)}(T_1) \over w}\biggr)^{\!2+2\beta_2} w\, \mathcal{N}(d_3(T_1,w,V_{T_1})) + \biggl({\hat{H}^{(2)}(T_1) \over w}\biggr)^{\!2\beta_2} K e^{-r(T-T_1)} \mathcal{N}(d_4(T_1,w,V_{T_1}))\Biggr] dw.
\end{aligned}
\]
Like in Subsection \ref{sec:firstorder}, let us rewrite this as
\[
\hat{f}_0^{(2)}(t',x) =  e^{-r(T_1-t')} \biggl( \mathbb{E}\Bigl[ F(W_1)\, \mathds{1}_{\{W_1 > \log \hat{H}^{(2)}(T_1)\}} \Bigr] - \Bigl({\hat{H}^{(2)}(t') \over x}\Bigr)^{2\beta_1} \mathbb{E}\Bigl[ F(W_2)\, \mathds{1}_{\{W_2 > \log \hat{H}^{(2)}(T_1)\}} \Bigr] \biggr)
\]
with $W_i \sim \text{Normal}\bigl(\mu_i, \gamma^2(t',T_1,v')\bigr)$ and
\begin{equation} \label{eq:ap_twostage_Ffun}
F(W) = \sum_{j=1}^4 a_j\, e^{\eta_{j\!}^{} W} \mathcal{N}\bigl(\nu_j W + \kappa_j \bigr)
\end{equation}
where $a_j, \eta_j, \nu_j, \kappa_j$ are the functions given in Table \ref{tab:params_ap_twostage_Ffun}. Combining this with Lemma \ref{lem:expect_closedform2}, we obtain a closed-form expression for $\hat{f}_0^{(2)}(t',x)$ in terms of the bivariate normal cumulative distribution function:
\begin{equation} \label{eq:ap_twostage_closed}
\begin{aligned}
\hat{f}_0^{(2)}(t',x) = e^{-r(T_1-t')} & \Biggl[\, \sum_{j=1}^4 a_j\, \Upsilon\bigl( \nu_j, \kappa_j, \eta_j; \mu_1, \gamma^2(t',T_1,v'), \hat{L} \bigr) \\
& \; - \Bigl({\hat{H}^{(2)}(t') \over x}\Bigr)^{2\beta_1} \sum_{j=1}^4 a_j\, \Upsilon\bigl( \nu_j, \kappa_j, \eta_j; \mu_2, \gamma^2(t',T_1,v'), \hat{L} \bigr) \Biggr]
\end{aligned}
\end{equation}
where $\hat{L} \equiv \log\hat{H}^{(2)}(T_1)$. (It can be verified numerically that this formula coincides with \eqref{eq:zeroord_approx_closed_beta_aw} in the particular case $\beta_1 = \beta_2 = \beta$, and also that it satisfies the standard monotonicity properties with respect to the barrier function.)

\begin{table*}[t!]
\centering
\normalsize
\caption{Parameters in equations \eqref{eq:ap_twostage_Ffun} and \eqref{eq:ap_twostage_closed}. The arguments of the functions $\gamma(T_1,T,V_{T_1})$ and $\hat{H}^{(2)}(T_1)$ are omitted.}
\def\arraystretch{1.45}
\begin{tabular}{ccccc}
\hline
$j$ & $a_j$ & $\eta_j$ & $\nu_j$ & $\kappa_j$ \\
\hline
$1$ & $1$ & $1$ & ${1 \over \gamma}$ & ${1 \over \gamma}\bigl[-\log(K \vee H_1) +r(T-T_1) + {\gamma^2 \over 2}\bigr]$\!\!\! \\
$2$ & $-K e^{-r(T-T_1)}$ & $0$ & ${1 \over \gamma}$ & ${1 \over \gamma}\bigl[-\log(K \vee H_1) +r(T-T_1) - {\gamma^2 \over 2}\bigr]$\!\!\! \\
$3$ & $-(\hat{H}^{(2)})^{2+2\beta_2}$ & $-(1+2\beta_2)$ & $-{1 \over \gamma}$ & ${1 \over \gamma}\bigl[\log\bigl({(\hat{H}^{(2)})^2 \over K \vee H_1}\bigr) + r(T-T_1) + {\gamma^2 \over 2}\bigr]$ \\
$4$ & $(\hat{H}^{(2)})^{2\beta_2} K e^{-r(T-T_1)}$ & $-2\beta_2$ & $-{1 \over \gamma}$ & ${1 \over \gamma}\bigl[\log\bigl({(\hat{H}^{(2)})^2 \over K \vee H_1}\bigr) + r(T-T_1) - {\gamma^2 \over 2}\bigr]$ \\
\hline
\end{tabular} 
\label{tab:params_ap_twostage_Ffun}
\end{table*}

For higher $n$, it is straightforwardly seen that, as a result of the successive $n-1$ integrations \eqref{eq:zeroord_multistage_iter}, the $n$-stage zero-order term $\hat{f}_0^{(n)}(t',x)$ can be written in terms of the cumulative distribution function of the $n$-dimensional normal distribution.

\renewcommand\section\apsection

\section*{References}

\bibliography{mybibfile2}

\end{document}